\documentclass[a4paper, 12pt]{amsart}

\usepackage{amssymb}
\usepackage{amsmath}
\usepackage{amscd}
\usepackage{graphicx}
\usepackage{amsthm}
\usepackage{color}
\usepackage{enumerate}
\usepackage[all]{xy}

\theoremstyle{definition}
\newtheorem{defn}{Definition}[section]
\newtheorem{thm}[defn]{Theorem}
\newtheorem{cor}[defn]{Corollary}
\newtheorem{prop}[defn]{Proposition}

\newtheorem{ex}[defn]{Example}

\newtheorem{conj}[defn]{Conjecture}
\newtheorem{rem}[defn]{Remark}

\newcommand{\mf}[1]{{\mathfrak{#1}}}
\newcommand{\mb}[1]{{\mathbf{#1}}}
\newcommand{\bb}[1]{{\mathbb{#1}}}
\newcommand{\mca}[1]{{\mathcal{#1}}}

\DeclareMathOperator{\rank}{rank}

\begin{document}

\title{Strong Koszulness of toric rings \\ associated with stable set polytopes \\ of trivially perfect graphs}
\author{Kazunori Matsuda}
\address[Kazunori Matsuda]{Department of Mathematics,  College of Science,  Rikkyo University, 
         Toshima-ku,  Tokyo 171-8501,  Japan}
\email{matsuda@rikkyo.ac.jp}
\date{\today}


\begin{abstract}
We give necessary and sufficient conditions for strong Koszulness of toric rings
associated with stable set polytopes of graphs. 
\end{abstract}

\maketitle


\section{Introduction} 

Let $G$ be a simple graph on the vertex set $V(G) = [n]$ with the edge set $E(G)$. 
$S \subset V(G)$ is said to be {\em stable} if $\{i,  j\} \not\in E(G)$ for all $i$,  $j \in S$. 
Note that $\emptyset$ is stable. 
For each stable set $S$ of $G$, we define $\rho(S) = \sum_{i \in S} {\bf e}_{i} \in \bb{R}^{n}$, where ${\bf e}_{i}$ is the $i$-th unit coordinate vector
in $\bb{R}^{n}$. 

The convex hull of $\{\rho(S) \mid S \text{\ is a stable set of\ } G \}$ is called the {\em stable set polytope} of $G$ (see \cite{C} ) , denoted by $\mathcal{Q}_{G}$. 
$\mathcal{Q}_{G}$ is a kind of $(0, 1)$-polytope. 
For this polytope,  we define the subring of $k[T,  X_{1},  \ldots,  X_{n}]$ as follows:

\begin{center}
$k[\mathcal{Q}_{G}] := k[T\cdot X_{1}^{a_{1}} \cdots X_{n}^{a_{n}} \mid (a_{1},  \ldots,  a_{n})  \text{\ is a vertex of } \mathcal{Q}_{G}]$, 
\end{center}

\noindent where $k$ is a field. 
$k[\mathcal{Q}_{G}]$ is called the {\em toric ring associated with the stable set polytope of } $G$.  
We can regard $k[\mathcal{Q}_{G}]$ as a graded $k$-algebra by setting $\deg T\cdot X_{1}^{a_{1}} \cdots X_{n}^{a_{n}} = 1$.  

In the theory of graded algebras, the notion of Koszulness (introduced by  Priddy \cite{P} ) plays an important role and is closely related to the Gr\"{o}bner basis theory. 

Let $\mathcal{P}$ be an integral convex polytope (i.e., a convex polytope each of whose vertices has integer coordinates) 
and $k[\mathcal{P}] := k[T\cdot X_{1}^{a_{1}} \cdots X_{n}^{a_{n}} \mid (a_{1},  \ldots,  a_{n})  \text{\ is a vertex of } \mathcal{P}]$ be the toric ring associated with $\mathcal{P}$. 
In general, it is known that


\begin{center}

The defining ideal of $k[\mathcal{P}]$ possesses a quadratic Gr\"{o}bner basis

$\Downarrow$

$k[\mathcal{P}]$ is Koszul 

$\Downarrow$

The defining ideal of $k[\mathcal{P}]$ is generated by quadratic binomials

\end{center}

\vspace{2mm}

\noindent follows from general theory (for example,  see \cite{BHeV}).

In this note,  we study the notion of a {\em strongly Koszul} algebra. 
In \cite{HeHiR},  Herzog,  Hibi, and Restuccia introduced this concept 
and discussed the basic properties of strongly Koszul algebras. 
Moreover,  they proposed the conjecture that the strong Koszulness of $R$ is at the top of the above hierarchy,  that is,

\begin{conj}[see \cite{HeHiR}]
The defining ideal of a strongly Koszul algebra $k[\mathcal{P}]$ possesses a quadratic Gr\"{o}bner basis. 
\end{conj}

A ring $R$ is {\em trivial} if $R$ can be constructed by 
starting from polynomial rings and repeatedly applying tensor and Segre products. 
In this note,  we propose the following conjecture. 

\begin{conj}
Let $\mathcal{P}$ be a $(0,  1)$-polytope and $k[\mathcal{P}]$ be the toric ring generated by $\mathcal{P}$. 
If $k[\mathcal{P}]$ is strongly Koszul,  then $k[\mathcal{P}]$ is trivial. 
\end{conj}

In the case of a $(0,  1)$-polytope,  Conjecture 1.2 implies Conjecture 1.1. 
If $\mathcal{P}$ is an order polytope or an edge polytope of bipartite graphs,  then Conjecture 1.2 holds true \cite{HeHiR}. 

In this note,  we prove Conjecture 1.2 for stable set polytopes. 
The main theorem of this note is the following:

\begin{thm}
Let $G$ be a graph. 
Then the following assertions are equivalent: 
\begin{enumerate}
	\item $k[\mathcal{Q}_{G}]$ is strongly Koszul.
	\item $G$ is a trivially perfect graph.
\end{enumerate}
In particular,  if $k[\mathcal{Q}_{G}]$ is strongly Koszul, then $k[\mathcal{Q}_{G}]$ is trivial. 
\end{thm}

Throughout this note,  we will use the standard terminologies of graph theory in \cite{Diest}.


\section{Strongly Koszul algebra}

Let $k$ be a field, $R$ be a graded $k$-algebra,  
and $\mf{m} = R_{+}$ be the homogeneous maximal ideal of $R$. 

\begin{defn}[\cite{HeHiR}]
A graded $k$-algebra $R$ is said to be {\em strongly Koszul}  if $\mf{m}$ admits a minimal system of generators $\{u_{1},  \ldots,  u_{t}\}$
which satisfies the following condition: 

\begin{quote}
For all  subsequences $u_{i_{1}},  \ldots,  u_{i_{r}}$ of $\{u_{1},  \ldots,  u_{t}\}$  $(i_{1} \le \cdots \le i_{r})$ and for all $j = 1,  \ldots,  r - 1$,  
$(u_{i_{1}},  \ldots,  u_{i_{j - 1}}) : u_{i_{j}}$ is generated by a subset of elements of $\{u_{1},  \ldots,  u_{t}\}$.
\end{quote}
\end{defn}

A graded $k$-algebra $R$ is called Koszul if $k = R/\mf{m}$ has a linear resolution.
By the following theorem,  we can see that a strongly Koszul algebra is Koszul. 

\begin{prop}[{[HeHiR,  Theorem 1.2]}]
If $R$ is strongly Koszul with respect to the minimal homogeneous generators $\{u_{1},  \ldots,  u_{t}\}$ of $\mf{m} = R_{+}$,  
then for all subsequences $\{u_{i_{1}},  \ldots,  u_{i_{r}}\}$ of $\{u_{1},  \ldots,  u_{t}\}$,  $R/(u_{i_{1}},  \ldots,  u_{i_{r}})$ has a linear resolution.
\end{prop}

The following proposition plays an important role in the proof of the main theorem. 

\begin{thm}[{[HeHiR,  Proposition 2.1]}]
Let $S$ be a semigroup and $R = k[S]$ be the semigroup ring generated by $S$. 
Let $\{u_{1},  \ldots,  u_{t}\}$ be the generators of $\mf{m} = R_{+}$ which correspond to the generators of $S$. 
Then,  if $R$ is strongly Koszul,  then for all subsequences $\{u_{i_{1}},  \ldots,  u_{i_{r}}\}$ of $\{u_{1},  \ldots,  u_{t}\}$,  
$R/(u_{i_{1}},  \ldots,  u_{i_{r}})$ is also strongly Koszul. 
\end{thm}

By this theorem,  we have

\begin{cor}
If $k[\mca{Q}_{G}]$ is strongly Koszul, then $k[\mca{Q}_{G_{W}}]$ is strongly Koszul for all induced subgraphs $G_{W}$  of $G$. 
\end{cor}


\section{Hibi ring and comparability graph}

In  this section,  we introduce the concepts of a Hibi ring and a comparability graph. 
Both are defined with respect to a partially ordered set. 

Let $P = \{p_{1},  \ldots,  p_{n}\}$ be a finite partially ordered set consisting of $n$ elements, which is referred to as a {\em poset}. 
Let $J(P)$ be the set of all poset ideals of $P$,  where a poset ideal of $P$ is a subset $I$ of $P$ 
such that if $x \in I$,  $y \in P$,  and $y \le x$, then $y \in I$.  
Note that $\emptyset \in J(P)$. 

First,  we give the definition of the Hibi ring introduced by Hibi. 

\begin{defn}[\cite{Hib}]
For a poset $P = \{p_{1},  \ldots,  p_{n}\}$,  the {\em Hibi ring} $\mca{R}_{k}[P]$ is defined as follows: 

\vspace{3mm}
\[
\mca{R}_{k}[P] := k[T\cdot \prod_{i \in I}X_{i} \mid I \in J(P)] \subset k[T,  X_{1},  \ldots,  X_{n}]
\]

\end{defn}

\begin{ex}
Consider the following poset $P = (1 \le 3,  2 \le 3$ and $2 \le 4)$.

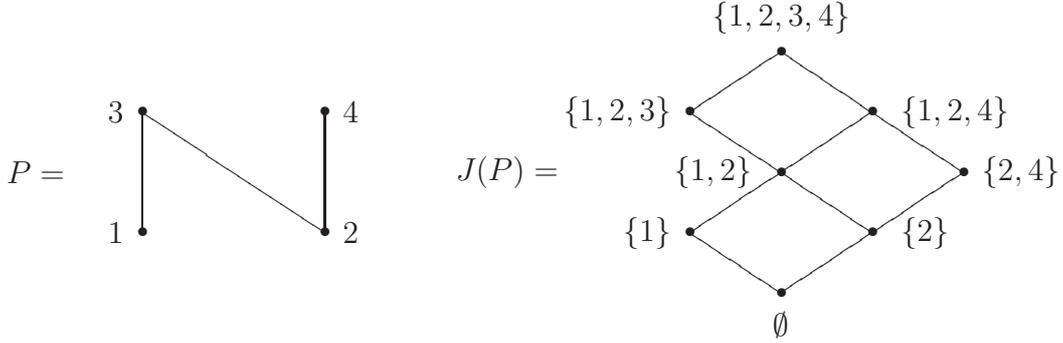
\vspace{2mm}

\begin{xy}
	\ar@{} (0,  0); (10,  -16) = "A", 
	\ar@{} "A" *{P = }; (24,  -24) *++!R{1} *\dir<4pt>{*} = "B", 
	\ar@{-} "B"; (24,  -8) *++!R{3} *\dir<4pt>{*} = "C", 
	\ar@{-} "C"; (48,  -24) *++!L{2} *\dir<4pt>{*} = "D", 
	\ar@{-} "D"; (48,  -8) *++!L{4} *\dir<4pt>{*} = "E", 
	\ar@{} "E"; (72,  -16) *{J(P) = } = "F", 
	\ar@{} "F"; (108,  -16) *++!R{\{1,  2\}\ } *\dir<4pt>{*} = "G", 
	\ar@{} "G"; (132,  -16) *++!L{\{2,  4\}} *\dir<4pt>{*} = "H", 
	\ar@{-} "G"; (120,  -8) *++!L{\ \{1,  2,  4\}} *\dir<4pt>{*} = "I", 
	\ar@{-} "H"; "I", 
	\ar@{} "I"; (120,  -24) *++!L{\ \{2\}} *\dir<4pt>{*} = "J", 
	\ar@{-} "G"; "J", 
	\ar@{-} "H"; "J", 
	\ar@{} "I",  (96,  -8) *++!R{\{1,  2,  3\}} *\dir<4pt>{*} = "K", 
	\ar@{-} "G"; "K", 
	\ar@{} "K",  (108,  0) *++!D{\{1,  2,  3,  4\}} *\dir<4pt>{*} = "L", 
	\ar@{-} "I"; "L", 
	\ar@{-} "K"; "L", 
	\ar@{-} "G"; (96,  -24) *++!R{\{1\}} *\dir<4pt>{*} = "M", 
	\ar@{-} "M"; (108,  -32) *++!U{\emptyset} *\dir<4pt>{*} = "N", 
	\ar@{-} "J"; "N", 
	\end{xy}

\noindent Then we have
\[
\mathcal{R}_{k}[P] = k[T,  TX_{1},  TX_{2},  TX_{1}X_{2},  TX_{2}X_{4},  TX_{1}X_{2}X_{3},  TX_{1}X_{2}X_{4},  TX_{1}X_{2}X_{3}X_{4}]. 
\]
\end{ex}

Hibi showed that a Hibi ring is always normal. 
Moreover,  a Hibi ring can be represented as a factor ring of a polynomial ring: if we let
\[
I_{P} := (X_{I}X_{J} - X_{I \cap J}X_{I \cup J} \mid I,  J \in J(P),  I \not\subseteq J \ \text{and} \ J \not\subseteq I) 
\]
be the binomial ideal in the polynomial ring $k[X_{I} \mid I \in J(P)]$ defined by a poset $P$,  then $\mca{R}_{k}[P] \cong k[X_{I} \mid I \in J(P)] / I_{P}$.  
Hibi also showed that $I_{P}$ has a quadratic Gr\"{o}bner basis for any term order which satisfies the following condition: the initial term of $X_{I}X_{J} - X_{I \cap J}X_{I \cup J}$ is 
$X_{I}X_{J}$. 
Hence a Hibi ring is always Koszul from general theory. 

Next,  we introduce the concept of a comparability graph. 

\begin{defn}
A graph $G$ is called a {\em comparability graph} if there exists a poset $P$ which satisfies the following condition:
\[
\{i,  j\} \in E(G) \iff i \ge j \hspace{3mm} \text{or} \hspace{3mm} i \le j \hspace{3mm} \text{in} \hspace{3mm} P. 
\]
We denote the comparability graph of $P$ by $G(P)$. 
\end{defn}

\begin{ex}

The lower-left poset $P$ defines the  comparability graph $G(P)$.

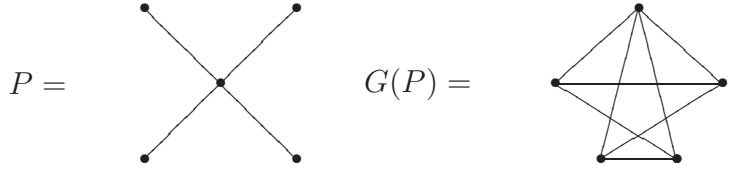
\vspace{8mm}

\begin{xy}
	\ar@{} (0,  0) ; (30,  10)  *{P = } = "A", 
	\ar@{} (0,  0) ; (44,  0)  *\dir<4pt>{*} = "B", 
	\ar@{-} "B"; (54,  10) *\dir<4pt>{*} = "C", 
	\ar@{-} "C"; (44,  20) *\dir<4pt>{*} = "D", 
	\ar@{-} "C"; (64,  0) *\dir<4pt>{*} = "E", 
	\ar@{-} "C"; (64,  20) *\dir<4pt>{*} = "F", 
	\ar@{} (0,  0) ; (80,  10)  *{G(P) = } = "G", 
	\ar@{} (0,  0) ; (104,  0) *\dir<4pt>{*} = "H", 
	\ar@{-} "H" ; (114,  0) *\dir<4pt>{*} = "I", 
	\ar@{} "H" ; (98,  10) *\dir<4pt>{*} = "J", 
	\ar@{-} "H" ; (120,  10) *\dir<4pt>{*} = "K", 
	\ar@{-} "H" ; (114,  0) *\dir<4pt>{*} = "I", 
	\ar@{-} "H" ; (109,  20) *\dir<4pt>{*} = "L", 
	\ar@{-} "I" ; "J", 
	\ar@{-} "J" ; "K", 
	\ar@{-} "J" ; "L", 
	\ar@{-} "I" ; "L", 
	\ar@{-} "L" ; "K", 
\end{xy}
\end{ex}

\vspace{3mm}

\begin{rem}
It is possible that $P \neq P^{'}$ but $G(P) = G(P^{'})$. 
Indeed,  for the following poset $P^{'}$,  $G(P^{'})$ is identical to $G(P)$ in the above example. 

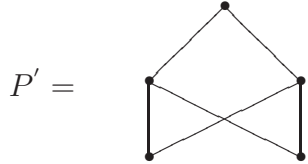
\vspace{5mm}

\begin{xy}
	\ar@{} (0,  0) ; (52,  10)  *{P^{'} = } = "A", 
	\ar@{} (0,  0) ; (66,  0)  *\dir<4pt>{*} = "B", 
	\ar@{} "B"; (76,  20) *\dir<4pt>{*} = "C", 
	\ar@{-} "C"; (66,  10) *\dir<4pt>{*} = "D", 
	\ar@{} "C"; (86,  0) *\dir<4pt>{*} = "E", 
	\ar@{-} "C"; (86,  10) *\dir<4pt>{*} = "F", 
	\ar@{-} "B" ; "D" ; 
	\ar@{-} "B" ; "F" ; 
	\ar@{-} "D" ; "E" ; 
	\ar@{-} "E" ; "F" ; 	
\end{xy}

\vspace{5mm}

\end{rem}

Complete graphs are comparability graphs of totally ordered sets.  
Bipartite graphs and trivially perfect graphs (see the next section) are also comparability graphs. 
Moreover,  if $G$ is a comparability graph, then the suspension (e.g., see [HiNOS,  p.4]) of $G$ is also a comparability graph. 

Recall the following definitions of two types of polytope which are defined by a poset.

\begin{defn}[see {[St1]}]
Let $P = \{p_{1},  \ldots,  p_{n}\}$ be a finite poset. 
	\begin{enumerate}
		\item The {\em order polytope} $\mca{O} (P)$ of $P$ is the convex polytope which consists of 
		$(a_{1},  \ldots,  a_{n}) \in \bb{R}^{n}$ such that $0 \le a_{i} \le 1$ with $a_{i} \ge a_{j}$ if $p_{i} \le p_{j}$ in $P$. 
		\item The {\em chain polytope} $\mca{C} (P)$ of $P$ is the convex polytope which consists of 
		$(a_{1},  \ldots,  a_{n}) \in \bb{R}^{n}$ such that $0 \le a_{i} \le 1$ with $a_{i_{1}} + \cdots + a_{i_{k}} \le 1$ 
		for all maximal chain $p_{i_{1}} < \cdots < p_{i_{k}}$ of $P$. 
	\end{enumerate}
\end{defn}

Let $\mca{C} (P)$ and $\mca{O} (P)$ be the chain polytope and order polytope of a finite poset $P$, respectively. 
In \cite{St1},  Stanley proved that 

\vspace{3mm}

\begin{center}
$\{\text{The vertices of\ } \mca{O} (P)\} = \{\rho(I) \mid I \text{\ is a poset ideal of\ } P\},$

\vspace{2mm}

\hspace{2mm} $\{\text{The vertices of\ } \mca{C} (P)\} = \{\rho(A) \mid A \text{\ is an anti-chain of\ } P\}$,  
\end{center}
\vspace{3mm}

\noindent where $A = \{p_{i_{1}},  \ldots,  p_{i_{k}}\}$ is an anti-chain of $P$ if $p_{i_{s}} \not\le p_{i_{t}}$ and $p_{i_{s}} \not\ge p_{i_{t}}$
for all $s \neq t$. 
Hence we have $\mca{Q}_{G(P)} = \mca{C} (P)$. 

In \cite{HiL},  Hibi and Li answered the question of when $\mca{C} (P)$ and $\mca{O} (P)$ are unimodularly equivalent. 
From their study,  we have the following theorem.

\begin{thm}[{[HiL,  Theorem 2.1]}]
Let $P$ be a poset and $G(P)$ be the comparability graph of $P$. 
Then the following are equivalent:
	\begin{enumerate}
		\item  The X-poset in Example 3.4 does not appear as a subposet (refer to [St2,  Chapter 3]) of $P$. 
		\item $\mca{R}_{k}[P] \cong k[\mca{Q}_{G(P)}]$.
	\end{enumerate} 
\end{thm}

\newpage

\begin{ex}
The cycle of length $4$ $C_{4}$  and the path of length $3$ $P_{4}$ are comparability graphs of $Q_{1}$ and $Q_{2}$, respectively. 

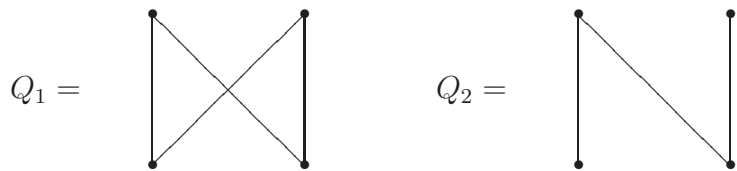
\vspace{8mm}

\begin{xy}
	\ar@{} (0,  0) ; (30,  0)  *{Q_{1} = } , 
	\ar@{} (0,  0) ; (44,  -10)  *\dir<4pt>{*} = "A", 
	\ar@{} "A"; (64,  -10) *\dir<4pt>{*} = "B", 
	\ar@{-} "A"; (44,  10) *\dir<4pt>{*} = "C", 
	\ar@{-} "B"; (64,  10) *\dir<4pt>{*} = "D", 
	\ar@{-} "B" ; "C", 
	\ar@{-} "A" ; "D", 
	\ar@{} (0,  0) ; (86,  0)  *{Q_{2} = } , 
	\ar@{} (0,  0) ; (100,  -10) *\dir<4pt>{*} = "E", 
	\ar@{} "E" ; (120,  -10) *\dir<4pt>{*} = "F", 
	\ar@{-} "E" ; (100,  10) *\dir<4pt>{*} = "G", 
	\ar@{-} "F" ; (120,  10) *\dir<4pt>{*} = "H", 
	\ar@{-} "F" ; "G", 
\end{xy}

\vspace{8mm}

\noindent Hence $k[\mca{Q}_{C_{4}}] \cong \mca{R}_{k}[Q_{1}]$ and $k[\mca{Q}_{P_{4}}] \cong \mca{R}_{k}[Q_{2}]$.

\end{ex}

A ring $R$ is {\em trivial} if $R$ can be constructed by 
starting from polynomial rings and repeatedly applying tensor and Segre products. 
Herzog, Hibi and Restuccia gave an answer for the question of when is a Hibi ring strongly Koszul. 

\begin{thm}[see {[HeHiR,  Theorem 3.2]}]
Let $P$ be a poset and $R = \mca{R}_{k}[P]$ be the Hibi ring constructed from $P$.
Then the following assertions are equivalent:
\begin{enumerate}
	\item $R$ is strongly Koszul.
	\item $R$ is trivial.
	\item The N-poset as described below does not appear as a subposet of $P$.  
\end{enumerate}

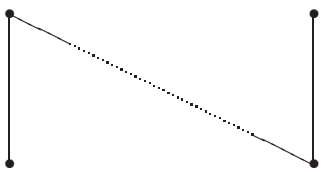
\vspace{5mm}

\begin{xy}
	\ar@{} (0,  0) ; (54,  -10)  *\dir<4pt>{*} = "A", 
	\ar@{-} "A" ; (54,  10)  *\dir<4pt>{*} = "B", 
	\ar@{-} "B" ; (62,  6)  *\dir<4pt>{} = "C", 
	\ar@{.} "C" ; (86,  -6)  *\dir<4pt>{} = "D", 
	\ar@{-} "D" ; (94,  -10)  *\dir<4pt>{*} = "E", 
	\ar@{-} "E" ; (94,  10)  *\dir<4pt>{*} = "F", 
\end{xy}

\vspace{5mm}

\end{thm}

By this theorem,  Corollary 2.4, and Example 3.8,  we have

\begin{cor}
If $G$ contains $C_{4}$ or $P_{4}$ as an induced subgraph,  then $k[\mca{Q}_{G}]$ is not strongly Koszul.
\end{cor}


\section{trivially perfect graph}

In this section,  we introduce the concept of a trivially perfect graph. 
As its name suggests, a trivially perfect graph is a kind of perfect graph; it is also a kind of comparability graph, as described below. 

\begin{defn}
For a graph $G$, we set 
\[
\alpha(G) := \max\{\#S \mid S \text{\ is a stable set of\ }G\},  
\]
\[
m(G) := \#\{ \text{the set of maximal cliques of } G\}. 
\]

\vspace{3mm}

We call $\alpha(G)$ the {\em stability number} (or {\em independence number}) of $G$.
\end{defn}

In general, $\alpha(G) \le m(G)$. 
Moreover,  if $G$ is chordal, then $m(G) \le n$ by Dirac's theorem \cite{Dir}.
In \cite{G},  Golumbic introduced the concept of a trivially perfect graph.

\begin{defn}[\cite{G}]
We say that a graph $G$ is {\em trivially perfect} if $\alpha(G_{W}) = m(G_{W})$ for any induced subgraph $G_{W}$ of $G$.  
\end{defn}

For example,  complete graphs and star graphs (i.e., the complete bipartite graph $K_{1,  r}$) are trivially perfect. 

We define some additional concepts related to perfect graphs. 
Let $C_{G}$ be the set of all cliques of $G$. 
Then we define 

\[
\omega(G) := \max\{\#C \mid C \in C_{G}\}, 
\]
\[
\theta(G) := \min\{s \mid C_{1} \coprod \cdots \coprod C_{s} = V(G),  C_{i} \in C_{G} \}, 
\]
\[
\chi(G) := \theta(\overline{G}), 
\]

\vspace{3mm}

\noindent where $\overline{G}$ is the complement of $G$. 
These invariants are called the {\em clique number},  {\em clique covering number}, and {\em chromatic number} of $G$, respectively.

In general,  $\alpha(G) = \omega(\overline{G})$, $\theta(G) \le m(G)$ and $\omega(G) \le \chi(G)$.
The definition of a perfect graph is as follows. 


\begin{defn}
We say that a graph $G$ is {\em perfect} if $\omega(G_{W}) = \chi(G_{W})$ for any induced subgraph $G_{W}$ of $G$.  
\end{defn}

Lov\'{a}sz proved that $G$ is perfect if and only if $\overline{G}$ is perfect \cite{Lo}. 
The theorem is now called the weak perfect graph theorem.
With it,  it is easy to show that a trivially perfect graph is perfect. 

\begin{prop}
Trivially perfect graphs are perfect. 
\end{prop}
\begin{proof}
Assume that $G$ is trivially perfect.
By \cite{Lo},  it is enough to show that $\overline{G}$ is perfect.
For all induced subgraphs $\overline{G}_{W}$ of $\overline{G}$, we have 
\[
m(G_{W}) = \alpha(G_{W}) = \omega(\overline{G_{W}}) \le \chi(\overline{G_{W}}) = \theta(G_{W}) \le m(G_{W})
\]

\noindent by general theory (note that $\overline{G}_{W} = \overline{G_{W}}$). 
\end{proof}

Golumbic gave a characterization of trivially perfect graphs.

\begin{thm}[{[G,  Theorem 2]}]
The following assertions are equivalent:
\begin{enumerate}
	\item $G$ is trivially perfect.
	\item $G$ is {\em $C_{4},  P_{4}$-free},  that is,  $G$ contains neither $C_{4}$ nor $P_{4}$ as an induced subgraph.
\end{enumerate}
\end{thm}
\begin{proof}
$(1) \Rightarrow (2)$:  
It is clear since $\alpha(C_{4}) = 2$,  $m(C_{4}) = 4$, and $\alpha(P_{4}) = 2$, $m(P_{4}) = 3$. 

$(2) \Rightarrow (1)$:  
Assume that $G$ contains neither $C_{4}$ nor $P_{4}$ as an induced subgraph.
If $G$ is not trivially perfect,  then there exists an induced subgraph $G_{W}$ of $G$
such that $\alpha(G_{W}) < m(G_{W})$. 
For this $G_{W}$,  there exists a maximal stable set $S_{W}$ of $G_{W}$ which satisfies the following: 

\vspace{2mm}

There exists $s \in S_{W}$ such that $s \in C_{1} \cap C_{2}$ for some distinct pair of cliques $C_{1}$,  $C_{2} \in C_{G_{W}}$. 

\vspace{2mm}

\noindent Note that $\#S_{W} > 1$ since $G_{W}$ is not complete.
Then there exist $x \in C_{1}$ and $y \in C_{2}$ such that $\{x,  s\}$,  $\{y,  s\} \in E(G_{W})$ and  $\{x,  y\} \not\in E(G_{W})$.

Let $u \in S_{W} \setminus \{s\}$.
If $\{x,  u\} \in E(G_{W})$ or $\{y,  u\} \in E(G_{W})$, then the induced graph $G_{\{x,  y,  s,  u\}}$ is $C_{4}$ or $P_{4}$, a contradiction.
Hence $\{x,  u\} \not\in E(G_{W})$ and $\{y,  u\} \not\in E(G_{W})$.
Then $\{x,  y\} \cup \{S \setminus \{s\}\}$ is a stable set of $G_{W}$, which contradicts that $S$ is maximal.
Therefore, $G$ is trivially perfect.
\end{proof}

Next,  we show that a trivially perfect graph is a kind of comparability graph. 
First,  we define the notion of a tree poset. 

\begin{defn}[see \cite{W}]
 A poset $P$ is a {\em tree} if it satisfies the following conditions:
\begin{enumerate}
	\item Each of the connected components of $P$ has a minimal element. 
	\item For all $p$,  $p^{'} \in P$,  the following assertion holds:  if there exists $q \in P$ such that $p,  p^{'} \le q$, then $p \le p^{'}$ or $p \ge p^{'}$.
\end{enumerate}
\end{defn}

\begin{ex}
The following poset is a tree: 

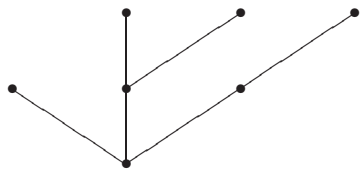
\vspace{8mm}

\begin{xy}
	\ar@{} (0,  0) ; (75,  0)  *\dir<4pt>{*} = "A", 
	\ar@{-} "A" ; (60,  10)  *\dir<4pt>{*} = "B", 
	\ar@{-} "A" ; (75,  10)  *\dir<4pt>{*} = "C", 
	\ar@{-} "A" ; (90,  10)  *\dir<4pt>{*} = "D", 
	\ar@{-} "C" ; (75,  20)  *\dir<4pt>{*} = "E", 
	\ar@{-} "C" ; (90,  20)  *\dir<4pt>{*} = "F", 
	\ar@{-} "D" ; (105,  20)  *\dir<4pt>{*} = "G", 	
\end{xy}

\vspace{5mm}

\end{ex}

Tree posets can be characterized as follows. 

\begin{prop}
Let $P$ be a poset. 
Then the following assertions are equivalent:
\begin{enumerate}
	\item $P$ is a tree.
	\item Neither the X-poset in Example 3.4,  the N-poset in Theorem 3.9, nor the diamond poset as described below 
	appears as a subposet of $P$. 
\end{enumerate}

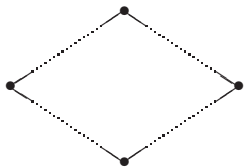
\vspace{5mm}

\begin{xy}
	\ar@{} (0,  0) ; (75,  0)  *\dir<4pt>{*} = "A", 
	\ar@{-} "A" ; (72,  2)  *\dir<4pt>{} = "B", 
	\ar@{.} "B" ; (63,  8)  *\dir<4pt>{} = "C", 
	\ar@{-} "C" ; (60,  10)  *\dir<4pt>{*} = "D", 
	\ar@{-} "A" ; (78,  2)  *\dir<4pt>{} = "E", 
	\ar@{.} "E" ; (87,  8)  *\dir<4pt>{} = "F", 
	\ar@{-} "F" ; (90,  10)  *\dir<4pt>{*} = "G", 
	\ar@{-} "D" ; (63,  12)  *\dir<4pt>{} = "H", 
	\ar@{.} "H" ; (72,  18)  *\dir<4pt>{} = "I", 
	\ar@{-} "I" ; (75,  20)  *\dir<4pt>{*} = "J", 
	\ar@{-} "G" ; (87,  12)  *\dir<4pt>{} = "K", 
	\ar@{.} "K" ; (78,  18)  *\dir<4pt>{} = "L", 
	\ar@{-} "L" ; "J", 
\end{xy}

\vspace{3mm}

\end{prop}

In \cite{W},  Wolk discussed the properties of the comparability graphs of a tree poset  
and showed that such graphs are exactly the graphs that satisfy the ``diagonal condition". 
This condition is equivalent to being $C_{4}$, $P_{4}$-free, and hence we have

\begin{cor}
Let $G$ be a graph. 
Then the following assertions are equivalent:
\begin{enumerate}
	\item $G$ is trivially perfect. 
	\item $G$ is a comparability graph of a tree poset. 
	\item $G$ is $C_{4}$, $P_{4}$-free.
\end{enumerate}
\end{cor}

\begin{rem}
A graph $G$ is a {\em threshold graph} if it can be constructed from a one-vertex graph by repeated applications of the following two operations: 
	\begin{enumerate}
		\item Add a single isolated vertex to the graph. 
		\item Take a suspension of the graph. 
	\end{enumerate}
\end{rem}

The concept of a threshold graph was introduced by Chv\'{a}tal and Hammer \cite{CHam}. 
They proved that $G$ is a threshold graph if and only if $G$ is $C_{4}$, $P_{4}$, $2K_{2}$-free. 
Hence a trivially perfect graph is also called a {\em quasi-threshold graph}.


\newpage
\section{Proof of Main theorem}

In this section,  we prove the main theorem. 

\begin{thm}
Let $G$ be a graph. 
Then the following assertions are equivalent: 
\begin{enumerate}
	\item $k[\mathcal{Q}_{G}]$ is strongly Koszul. 
	\item $G$ is trivially perfect.
\end{enumerate}
\end{thm}
\begin{proof}
We assume that $G$ is trivially perfect. 
Then there exists a tree poset $P$ such that $G = G(P)$ 
from Corollary 4.9. 
This implies that neither the X-poset in Example 3.4 nor the N-poset in Theorem 3.9 appears as a subposet of $P$ by Proposition 4.8 ,  and hence 
$k[\mca{Q}_{G(P)}] \cong \mca{R}_{k}[P]$ is strongly Koszul by Theorems 3.7 and 3.9.  

Conversely,   if $G$ is not trivially perfect,  $G$ contains $C_{4}$ or $P_{4}$ as an induced subgraph by Corollary 4.9. 
Therefore,  we have that $k[\mca{Q}_{G}]$ is not strongly Koszul by Corollary 3.10. 
\end{proof}

%

\section{Remark on usual Koszulness of $k[\mca{Q}_{G}]$}

It seems to be difficult to give a complete characterization of when $k[\mca{Q}_{G}]$ is Koszul. 
However,  it is known that $k[\mca{Q}_{G}]$ is Koszul for many graphs $G$. 

\begin{thm}[\cite{EN}]
If $G$ is an almost bipartite graph,  then $k[\mca{Q}_{G}]$ is Koszul,  where a graph $G$ is almost bipartite
if there exists a vertex $v \in [n]$ such that the induced subgraph $G_{[n] \setminus v}$ is bipartite,  that is,  does not contain 
induced odd cycles. 
\end{thm}

\begin{rem} 
An almost bipartite graph is one such that all its odd cycles share a common vertex. 
Hence if $G$ is almost bipartite, then $G$ is $K_{4}$-free,  that is,  $\omega(G) \le 3$. 
In the case of $n \le 5$,  $G$ is almost bipartite if and only if $G$ is $K_{4}$-free. 
\end{rem}

Next,  we recall the theorem of Hibi and Li (Theorem 3.7).  

A graph $G$ is an {\em HL-comparability} graph if it is the comparability graph of a poset $P$ which 
does not contain the X-poset in Example 3.4 as a subposet. 

From their theorem,  we have 

\begin{thm}
If $G$ is an HL-comparability graph,  then $k[\mca{Q}_{G}]$ is Koszul. 
\end{thm}

\begin{rem}
	\begin{enumerate}
		\item If $n \le 5$,  the notion of HL-comparability is equivalent to the usual comparability. 
		\item Bipartite graphs are comparability graphs defined by posets with $\rank P \le 1$. 
		Hence bipartite graphs are HL-comparability graphs.  
		\item Let $G$ be a complete $r$-partite graph with $V(G) = \prod_{i = 1}^{r} V_{i}$. 
		Then $G$ is an HL-comparability graph if and only if $\#\{V_{i} \mid \#V_{i} = 1\} \ge r - 2$. 
		\item Let $G$ be a closed graph (see \cite{HeHiHrKR}) which satisfies the following condition: 
		for all $C_{1},  C_{2} \in C_{G}$,  $\#\{C_{1} \cap C_{2}\} \le 1$. 
		Then $G$ is an HL-comparability graph. 
	\end{enumerate}
\end{rem}

As the end of this note,  we give a classification table of connected six-vertex graphs using Harary \cite{Har}.

\newpage


\begin{xy}
	\ar@{} (0,  0);(100,  0) *\txt{Classification - six-vertex (112 items)};
	\ar@{-} (0,  0);(65,  0);
	\ar@{-} (134,  0);(150,  0) = "A";
	\ar@{-} (0,  0);(0,  -220) = "H";
	\ar@{-} "A";(150,  -220) = "I";
	\ar@{-} "H";"I";
	
	\textcolor{green}{\ar@{-} (8,  -6);(20,  -6);}
	\textcolor{green}{\ar@{} (0,  0);(36,  -6) *\txt{Almost Bipartite};}
	\textcolor{green}{\ar@{-} (52,  -6);(85,  -6);}
	\textcolor{green}{\ar@{-} (4,  -6);(4,  -182);}
	\textcolor{green}{\ar@{-} (82,  -6);(82,  -182);}
	\textcolor{green}{\ar@{-} (1,  -182);(81,  -182);}

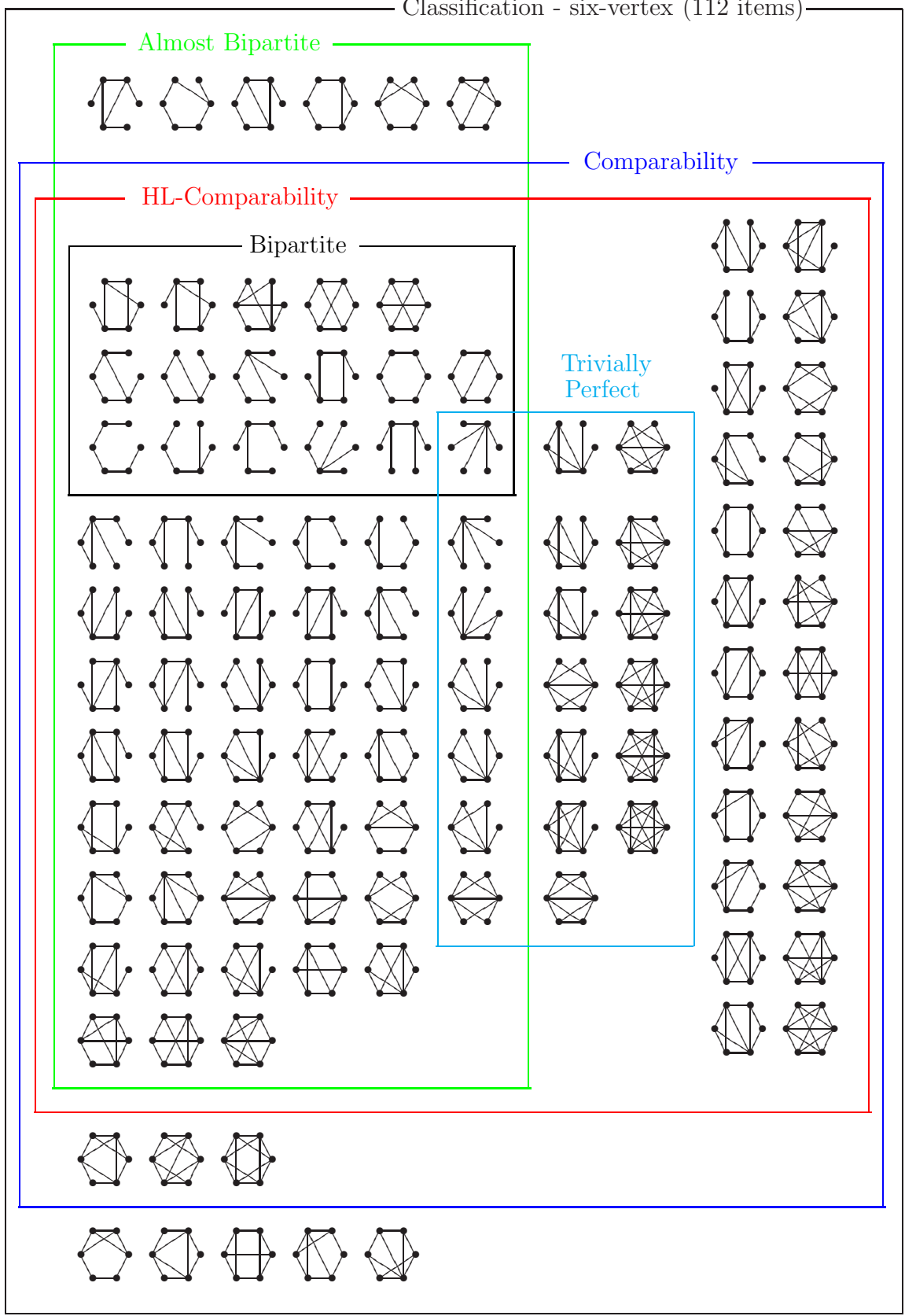
	\ar@{} (0,  0);(8,  -12) *\dir<2pt>{*} = "B13";
	\ar@{-} "B13";(6,  -16) *\dir<2pt>{*} = "C13";
	\ar@{} "C13";(8,  -20) *\dir<2pt>{*} = "D13";
	\ar@{-} "D13";(12,  -20) *\dir<2pt>{*} = "E13";
	\ar@{} "E13";(14,  -16) *\dir<2pt>{*} = "F13";
	\ar@{-} "F13";(12,  -12) *\dir<2pt>{*} = "G13";
	\ar@{-} "B13";"G13";
	\ar@{-} "B13";"D13";
	\ar@{} "B13";"E13";
	\ar@{} "B13";"F13";
	\ar@{} "C13";"E13";
	\ar@{} "C13";"F13";
	\ar@{} "C13";"G13";
	\ar@{} "D13";"F13";
	\ar@{-} "D13";"G13";
	\ar@{} "E13";"G13";

	\ar@{} (0,  0);(20,  -12) *\dir<2pt>{*} = "B18";
	\ar@{-} "B18";(18,  -16) *\dir<2pt>{*} = "C18";
	\ar@{-} "C18";(20,  -20) *\dir<2pt>{*} = "D18";
	\ar@{-} "D18";(24,  -20) *\dir<2pt>{*} = "E18";
	\ar@{-} "E18";(26,  -16) *\dir<2pt>{*} = "F18";
	\ar@{-} "F18";(24,  -12) *\dir<2pt>{*} = "G18";
	\ar@{} "B18";"G18";
	\ar@{} "B18";"D18";
	\ar@{} "B18";"E18";
	\ar@{-} "B18";"F18";
	\ar@{} "C18";"E18";
	\ar@{} "C18";"F18";
	\ar@{} "C18";"G18";
	\ar@{} "D18";"F18";
	\ar@{} "D18";"G18";
	\ar@{} "E18";"G18";

	\ar@{} (0,  0);(32,  -12) *\dir<2pt>{*} = "B35";
	\ar@{-} "B35";(30,  -16) *\dir<2pt>{*} = "C35";
	\ar@{-} "C35";(32,  -20) *\dir<2pt>{*} = "D35";
	\ar@{-} "D35";(36,  -20) *\dir<2pt>{*} = "E35";
	\ar@{} "E35";(38,  -16) *\dir<2pt>{*} = "F35";
	\ar@{-} "F35";(36,  -12) *\dir<2pt>{*} = "G35";
	\ar@{-} "B35";"G35";
	\ar@{} "B35";"D35";
	\ar@{-} "B35";"E35";
	\ar@{} "B35";"F35";
	\ar@{} "C35";"E35";
	\ar@{} "C35";"F35";
	\ar@{} "C35";"G35";
	\ar@{} "D35";"F35";
	\ar@{} "D35";"G35";
	\ar@{-} "E35";"G35";

	\ar@{} (0,  0);(44,  -12) *\dir<2pt>{*} = "B38";
	\ar@{-} "B38";(42,  -16) *\dir<2pt>{*} = "C38";
	\ar@{-} "C38";(44,  -20) *\dir<2pt>{*} = "D38";
	\ar@{-} "D38";(48,  -20) *\dir<2pt>{*} = "E38";
	\ar@{-} "E38";(50,  -16) *\dir<2pt>{*} = "F38";
	\ar@{-} "F38";(48,  -12) *\dir<2pt>{*} = "G38";
	\ar@{-} "B38";"G38";
	\ar@{} "B38";"D38";
	\ar@{} "B38";"E38";
	\ar@{} "B38";"F38";
	\ar@{} "C38";"E38";
	\ar@{} "C38";"F38";
	\ar@{} "C38";"G38";
	\ar@{} "D38";"F38";
	\ar@{} "D38";"G38";
	\ar@{-} "E38";"G38";

	\ar@{} (0,  0);(56,  -12) *\dir<2pt>{*} = "B39";
	\ar@{-} "B39";(54,  -16) *\dir<2pt>{*} = "C39";
	\ar@{-} "C39";(56,  -20) *\dir<2pt>{*} = "D39";
	\ar@{-} "D39";(60,  -20) *\dir<2pt>{*} = "E39";
	\ar@{-} "E39";(62,  -16) *\dir<2pt>{*} = "F39";
	\ar@{-} "F39";(60,  -12) *\dir<2pt>{*} = "G39";
	\ar@{} "B39";"G39";
	\ar@{} "B39";"D39";
	\ar@{} "B39";"E39";
	\ar@{-} "B39";"F39";
	\ar@{} "C39";"E39";
	\ar@{} "C39";"F39";
	\ar@{-} "C39";"G39";
	\ar@{} "D39";"F39";
	\ar@{} "D39";"G39";
	\ar@{} "E39";"G39";

	\ar@{} (0,  0);(68,  -12) *\dir<2pt>{*} = "B58";
	\ar@{-} "B58";(66,  -16) *\dir<2pt>{*} = "C58";
	\ar@{-} "C58";(68,  -20) *\dir<2pt>{*} = "D58";
	\ar@{-} "D58";(72,  -20) *\dir<2pt>{*} = "E58";
	\ar@{-} "E58";(74,  -16) *\dir<2pt>{*} = "F58";
	\ar@{-} "F58";(72,  -12) *\dir<2pt>{*} = "G58";
	\ar@{-} "B58";"G58";
	\ar@{} "B58";"D58";
	\ar@{} "B58";"E58";
	\ar@{-} "B58";"F58";
	\ar@{} "C58";"E58";
	\ar@{} "C58";"F58";
	\ar@{} "C58";"G58";
	\ar@{} "D58";"F58";
	\ar@{-} "D58";"G58";
	\ar@{} "E58";"G58";

	\textcolor{blue}{\ar@{-} (-6,  -26);(86,  -26);}
	\textcolor{blue}{\ar@{} (0,  0);(100,  -26) *\txt{Comparability};}
	\textcolor{blue}{\ar@{-} (114,  -26);(136,  -26);}
	\textcolor{blue}{\ar@{-} (-10,  -26);(-10,  -202);}
	\textcolor{blue}{\ar@{-} (133,  -26);(133,  -202);}
	\textcolor{blue}{\ar@{-} (-13,  -202);(132,  -202);}

	\ar@{} (0,  0);(0,  -46) *\dir<2pt>{*} = "B33";
	\ar@{} "B33";(-2,  -50) *\dir<2pt>{*} = "C33";
	\ar@{-} "C33";(0,  -54) *\dir<2pt>{*} = "D33";
	\ar@{-} "D33";(4,  -54) *\dir<2pt>{*} = "E33";
	\ar@{-} "E33";(6,  -50) *\dir<2pt>{*} = "F33";
	\ar@{} "F33";(4,  -46) *\dir<2pt>{*} = "G33";
	\ar@{-} "B33";"G33";
	\ar@{-} "B33";"D33";
	\ar@{} "B33";"E33";
	\ar@{-} "B33";"F33";
	\ar@{} "C33";"E33";
	\ar@{} "C33";"F33";
	\ar@{} "C33";"G33";
	\ar@{} "D33";"F33";
	\ar@{} "D33";"G33";
	\ar@{-} "E33";"G33";

	\ar@{} (0,  0);(12,  -46) *\dir<2pt>{*} = "B36";
	\ar@{-} "B36";(10,  -50) *\dir<2pt>{*} = "C36";
	\ar@{} "C36";(12 ,  -54) *\dir<2pt>{*} = "D36";
	\ar@{-} "D36";(16,  -54) *\dir<2pt>{*} = "E36";
	\ar@{-} "E36";(18,  -50) *\dir<2pt>{*} = "F36";
	\ar@{} "F36";(16,  -46) *\dir<2pt>{*} = "G36";
	\ar@{-} "B36";"G36";
	\ar@{-} "B36";"D36";
	\ar@{} "B36";"E36";
	\ar@{-} "B36";"F36";
	\ar@{} "C36";"E36";
	\ar@{} "C36";"F36";
	\ar@{} "C36";"G36";
	\ar@{} "D36";"F36";
	\ar@{} "D36";"G36";
	\ar@{-} "E36";"G36";

	\ar@{} (0,  0);(24,  -46) *\dir<2pt>{*} = "B49";
	\ar@{-} "B49";(22,  -50) *\dir<2pt>{*} = "C49";
	\ar@{-} "C49";(24,  -54) *\dir<2pt>{*} = "D49";
	\ar@{-} "D49";(28,  -54) *\dir<2pt>{*} = "E49";
	\ar@{-} "E49";(30,  -50) *\dir<2pt>{*} = "F49";
	\ar@{} "F49";(28,  -46) *\dir<2pt>{*} = "G49";
	\ar@{} "B49";"G49";
	\ar@{} "B49";"D49";
	\ar@{-} "B49";"E49";
	\ar@{} "B49";"F49";
	\ar@{} "C49";"E49";
	\ar@{-} "C49";"F49";
	\ar@{-} "C49";"G49";
	\ar@{} "D49";"F49";
	\ar@{} "D49";"G49";
	\ar@{-} "E49";"G49";

	\ar@{} (0,  0);(36,  -46) *\dir<2pt>{*} = "B57";
	\ar@{-} "B57";(34,  -50) *\dir<2pt>{*} = "C57";
	\ar@{-} "C57";(36,  -54) *\dir<2pt>{*} = "D57";
	\ar@{-} "D57";(40,  -54) *\dir<2pt>{*} = "E57";
	\ar@{-} "E57";(42,  -50) *\dir<2pt>{*} = "F57";
	\ar@{-} "F57";(40,  -46) *\dir<2pt>{*} = "G57";
	\ar@{-} "B57";"G57";
	\ar@{} "B57";"D57";
	\ar@{-} "B57";"E57";
	\ar@{} "B57";"F57";
	\ar@{} "C57";"E57";
	\ar@{} "C57";"F57";
	\ar@{} "C57";"G57";
	\ar@{} "D57";"F57";
	\ar@{-} "D57";"G57";
	\ar@{} "E57";"G57";

	\ar@{} (0,  0);(48,  -46) *\dir<2pt>{*} = "B69";
	\ar@{-} "B69";(46,  -50) *\dir<2pt>{*} = "C69";
	\ar@{-} "C69";(48,  -54) *\dir<2pt>{*} = "D69";
	\ar@{-} "D69";(52,  -54) *\dir<2pt>{*} = "E69";
	\ar@{-} "E69";(54,  -50) *\dir<2pt>{*} = "F69";
	\ar@{-} "F69";(52,  -46) *\dir<2pt>{*} = "G69";
	\ar@{-} "B69";"G69";
	\ar@{} "B69";"D69";
	\ar@{-} "B69";"E69";
	\ar@{} "B69";"F69";
	\ar@{} "C69";"E69";
	\ar@{-} "C69";"F69";
	\ar@{} "C69";"G69";
	\ar@{} "D69";"F69";
	\ar@{-} "D69";"G69";
	\ar@{} "E69";"G69";

	\ar@{} (0,  0);(0,  -58) *\dir<2pt>{*} = "B14";
	\ar@{-} "B14";(-2,  -62) *\dir<2pt>{*} = "C14";
	\ar@{-} "C14";(0,  -66) *\dir<2pt>{*} = "D14";
	\ar@{-} "D14";(4,  -66) *\dir<2pt>{*} = "E14";
	\ar@{-} "E14";(6,  -62) *\dir<2pt>{*} = "F14";
	\ar@{} "F14";(4,  -58) *\dir<2pt>{*} = "G14";
	\ar@{-} "B14";"G14";
	\ar@{} "B14";"D14";
	\ar@{-} "B14";"E14";
	\ar@{} "B14";"F14";
	\ar@{} "C14";"E14";
	\ar@{} "C14";"F14";
	\ar@{} "C14";"G14";
	\ar@{} "D14";"F14";
	\ar@{} "D14";"G14";
	\ar@{} "E14";"G14";
	
	\ar@{} (0,  0);(12,  -58) *\dir<2pt>{*} = "B15";
	\ar@{-} "B15";(10,  -62) *\dir<2pt>{*} = "C15";
	\ar@{-} "C15";(12,  -66) *\dir<2pt>{*} = "D15";
	\ar@{-} "D15";(16,  -66) *\dir<2pt>{*} = "E15";
	\ar@{-} "E15";(18,  -62) *\dir<2pt>{*} = "F15";
	\ar@{-} "F15";(16,  -58) *\dir<2pt>{*} = "G15";
	\ar@{} "B15";"G15";
	\ar@{} "B15";"D15";
	\ar@{-} "B15";"E15";
	\ar@{} "B15";"F15";
	\ar@{} "C15";"E15";
	\ar@{} "C15";"F15";
	\ar@{} "C15";"G15";
	\ar@{} "D15";"F15";
	\ar@{} "D15";"G15";
	\ar@{} "E15";"G15";

	\ar@{} (0,  0);(24,  -58) *\dir<2pt>{*} = "B16";
	\ar@{-} "B16";(22,  -62) *\dir<2pt>{*} = "C16";
	\ar@{-} "C16";(24,  -66) *\dir<2pt>{*} = "D16";
	\ar@{-} "D16";(28,  -66) *\dir<2pt>{*} = "E16";
	\ar@{} "E16";(30,  -62) *\dir<2pt>{*} = "F16";
	\ar@{} "F16";(28,  -58) *\dir<2pt>{*} = "G16";
	\ar@{-} "B16";"G16";
	\ar@{} "B16";"D16";
	\ar@{-} "B16";"E16";
	\ar@{-} "B16";"F16";
	\ar@{} "C16";"E16";
	\ar@{} "C16";"F16";
	\ar@{} "C16";"G16";
	\ar@{} "D16";"F16";
	\ar@{} "D16";"G16";
	\ar@{} "E16";"G16";
	
	\ar@{} (0,  0);(36,  -58) *\dir<2pt>{*} = "B17";
	\ar@{} "B17";(34,  -62) *\dir<2pt>{*} = "C17";
	\ar@{-} "C17";(36,  -66) *\dir<2pt>{*} = "D17";
	\ar@{-} "D17";(40,  -66) *\dir<2pt>{*} = "E17";
	\ar@{} "E17";(42,  -62) *\dir<2pt>{*} = "F17";
	\ar@{-} "F17";(40,  -58) *\dir<2pt>{*} = "G17";
	\ar@{-} "B17";"G17";
	\ar@{-} "B17";"D17";
	\ar@{} "B17";"E17";
	\ar@{} "B17";"F17";
	\ar@{} "C17";"E17";
	\ar@{} "C17";"F17";
	\ar@{} "C17";"G17";
	\ar@{} "D17";"F17";
	\ar@{} "D17";"G17";
	\ar@{-} "E17";"G17";

	\ar@{} (0,  0);(48,  -58) *\dir<2pt>{*} = "B19";
	\ar@{-} "B19";(46,  -62) *\dir<2pt>{*} = "C19";
	\ar@{-} "C19";(48,  -66) *\dir<2pt>{*} = "D19";
	\ar@{-} "D19";(52,  -66) *\dir<2pt>{*} = "E19";
	\ar@{-} "E19";(54,  -62) *\dir<2pt>{*} = "F19";
	\ar@{-} "F19";(52,  -58) *\dir<2pt>{*} = "G19";
	\ar@{-} "B19";"G19";
	\ar@{} "B19";"D19";
	\ar@{} "B19";"E19";
	\ar@{} "B19";"F19";
	\ar@{} "C19";"E19";
	\ar@{} "C19";"F19";
	\ar@{} "C19";"G19";
	\ar@{} "D19";"F19";
	\ar@{} "D19";"G19";
	\ar@{} "E19";"G19";

	\ar@{} (0,  0);(60,  -58) *\dir<2pt>{*} = "B31";
	\ar@{-} "B31";(58,  -62) *\dir<2pt>{*} = "C31";
	\ar@{-} "C31";(60,  -66) *\dir<2pt>{*} = "D31";
	\ar@{-} "D31";(64,  -66) *\dir<2pt>{*} = "E31";
	\ar@{-} "E31";(66,  -62) *\dir<2pt>{*} = "F31";
	\ar@{-} "F31";(64,  -58) *\dir<2pt>{*} = "G31";
	\ar@{-} "B31";"G31";
	\ar@{} "B31";"D31";
	\ar@{} "B31";"E31";
	\ar@{} "B31";"F31";
	\ar@{} "C31";"E31";
	\ar@{} "C31";"F31";
	\ar@{} "C31";"G31";
	\ar@{} "D31";"F31";
	\ar@{-} "D31";"G31";
	\ar@{} "E31";"G31";

	\ar@{} (0,  0);(0,  -70) *\dir<2pt>{*} = "B1";
	\ar@{-} "B1";(-2,  -74) *\dir<2pt>{*} = "C1";
	\ar@{-} "C1";(0,  -78) *\dir<2pt>{*} = "D1";
	\ar@{-} "D1";(4,  -78) *\dir<2pt>{*} = "E1";
	\ar@{-} "E1";(6,  -74) *\dir<2pt>{*} = "F1";
	\ar@{} "F1";(4,  -70) *\dir<2pt>{*} = "G1";
	\ar@{-} "B1";"G1";
	\ar@{} "B1";"D1";
	\ar@{} "B1";"E1";
	\ar@{} "B1";"F1";
	\ar@{} "C1";"E1";
	\ar@{} "C1";"F1";
	\ar@{} "C1";"G1";
	\ar@{} "D1";"F1";
	\ar@{} "D1";"G1";
	\ar@{} "E1";"G1";

	\ar@{} (0,  0);(12,  -70) *\dir<2pt>{*} = "B2";
	\ar@{-} "B2";(10,  -74) *\dir<2pt>{*} = "C2";
	\ar@{-} "C2";(12,  -78) *\dir<2pt>{*} = "D2";
	\ar@{-} "D2";(16,  -78) *\dir<2pt>{*} = "E2";
	\ar@{-} "E2";(18,  -74) *\dir<2pt>{*} = "F2";
	\ar@{} "F2";(16,  -70) *\dir<2pt>{*} = "G2";
	\ar@{} "B2";"G2";
	\ar@{} "B2";"D2";
	\ar@{} "B2";"E2";
	\ar@{} "B2";"F2";
	\ar@{} "C2";"E2";
	\ar@{} "C2";"F2";
	\ar@{} "C2";"G2";
	\ar@{} "D2";"F2";
	\ar@{} "D2";"G2";
	\ar@{-} "E2";"G2";

	\ar@{} (0,  0);(24,  -70) *\dir<2pt>{*} = "B3";
	\ar@{-} "B3";(22,  -74) *\dir<2pt>{*} = "C3";
	\ar@{} "C3";(24,  -78) *\dir<2pt>{*} = "D3";
	\ar@{-} "D3";(28,  -78) *\dir<2pt>{*} = "E3";
	\ar@{} "E3";(30,  -74) *\dir<2pt>{*} = "F3";
	\ar@{-} "F3";(28,  -70) *\dir<2pt>{*} = "G3";
	\ar@{-} "B3";"G3";
	\ar@{-} "B3";"D3";
	\ar@{} "B3";"E3";
	\ar@{} "B3";"F3";
	\ar@{} "C3";"E3";
	\ar@{} "C3";"F3";
	\ar@{} "C3";"G3";
	\ar@{} "D3";"F3";
	\ar@{} "D3";"G3";
	\ar@{} "E3";"G3";

	\ar@{} (0,  0);(36,  -70) *\dir<2pt>{*} = "B4";
	\ar@{-} "B4";(34,  -74) *\dir<2pt>{*} = "C4";
	\ar@{-} "C4";(36,  -78) *\dir<2pt>{*} = "D4";
	\ar@{-} "D4";(40,  -78) *\dir<2pt>{*} = "E4";
	\ar@{} "E4";(42,  -74) *\dir<2pt>{*} = "F4";
	\ar@{} "F4";(40,  -70) *\dir<2pt>{*} = "G4";
	\ar@{} "B4";"G4";
	\ar@{} "B4";"D4";
	\ar@{} "B4";"E4";
	\ar@{} "B4";"F4";
	\ar@{} "C4";"E4";
	\ar@{} "C4";"F4";
	\ar@{} "C4";"G4";
	\ar@{-} "D4";"F4";
	\ar@{-} "D4";"G4";
	\ar@{} "E4";"G4";

	\ar@{} (0,  0);(48,  -70) *\dir<2pt>{*} = "B5";
	\ar@{-} "B5";(46,  -74) *\dir<2pt>{*} = "C5";
	\ar@{} "C5";(48,  -78) *\dir<2pt>{*} = "D5";
	\ar@{} "D5";(52,  -78) *\dir<2pt>{*} = "E5";
	\ar@{} "E5";(54,  -74) *\dir<2pt>{*} = "F5";
	\ar@{-} "F5";(52,  -70) *\dir<2pt>{*} = "G5";
	\ar@{-} "B5";"G5";
	\ar@{-} "B5";"D5";
	\ar@{} "B5";"E5";
	\ar@{} "B5";"F5";
	\ar@{} "C5";"E5";
	\ar@{} "C5";"F5";
	\ar@{} "C5";"G5";
	\ar@{} "D5";"F5";
	\ar@{} "D5";"G5";
	\ar@{-} "E5";"G5";

	\ar@{} (0,  0);(60,  -70) *\dir<2pt>{*} = "B6";
	\ar@{} "B6";(58,  -74) *\dir<2pt>{*} = "C6";
	\ar@{} "C6";(60,  -78) *\dir<2pt>{*} = "D6";
	\ar@{} "D6";(64,  -78) *\dir<2pt>{*} = "E6";
	\ar@{} "E6";(66,  -74) *\dir<2pt>{*} = "F6";
	\ar@{-} "F6";(64,  -70) *\dir<2pt>{*} = "G6";
	\ar@{-} "B6";"G6";
	\ar@{} "B6";"D6";
	\ar@{} "B6";"E6";
	\ar@{} "B6";"F6";
	\ar@{} "C6";"E6";
	\ar@{} "C6";"F6";
	\ar@{-} "C6";"G6";
	\ar@{} "D6";"F6";
	\ar@{-} "D6";"G6";
	\ar@{-} "E6";"G6";

	\ar@{} (0,  0);(60,  -86) *\dir<2pt>{*} = "B7";
	\ar@{-} "B7";(58,  -90) *\dir<2pt>{*} = "C7";
	\ar@{-} "C7";(60,  -94) *\dir<2pt>{*} = "D7";
	\ar@{} "D7";(64,  -94) *\dir<2pt>{*} = "E7";
	\ar@{} "E7";(66,  -90) *\dir<2pt>{*} = "F7";
	\ar@{} "F7";(64,  -86) *\dir<2pt>{*} = "G7";
	\ar@{-} "B7";"G7";
	\ar@{-} "B7";"D7";
	\ar@{-} "B7";"E7";
	\ar@{-} "B7";"F7";
	\ar@{} "C7";"E7";
	\ar@{} "C7";"F7";
	\ar@{} "C7";"G7";
	\ar@{} "D7";"F7";
	\ar@{} "D7";"G7";
	\ar@{} "E7";"G7";

	\ar@{} (0,  0);(60,  -98) *\dir<2pt>{*} = "B27";
	\ar@{-} "B27";(58,  -102) *\dir<2pt>{*} = "C27";
	\ar@{-} "C27";(60,  -106) *\dir<2pt>{*} = "D27";
	\ar@{-} "D27";(64,  -106) *\dir<2pt>{*} = "E27";
	\ar@{-} "E27";(66,  -102) *\dir<2pt>{*} = "F27";
	\ar@{} "F27";(64,  -98) *\dir<2pt>{*} = "G27";
	\ar@{} "B27";"G27";
	\ar@{-} "B27";"D27";
	\ar@{} "B27";"E27";
	\ar@{} "B27";"F27";
	\ar@{} "C27";"E27";
	\ar@{} "C27";"F27";
	\ar@{} "C27";"G27";
	\ar@{-} "D27";"F27";
	\ar@{-} "D27";"G27";
	\ar@{} "E27";"G27";

	\ar@{} (0,  0);(60,  -110) *\dir<2pt>{*} = "B28";
	\ar@{-} "B28";(58,  -114) *\dir<2pt>{*} = "C28";
	\ar@{-} "C28";(60,  -118) *\dir<2pt>{*} = "D28";
	\ar@{-} "D28";(64,  -118) *\dir<2pt>{*} = "E28";
	\ar@{-} "E28";(66,  -114) *\dir<2pt>{*} = "F28";
	\ar@{} "F28";(64,  -110) *\dir<2pt>{*} = "G28";
	\ar@{} "B28";"G28";
	\ar@{} "B28";"D28";
	\ar@{-} "B28";"E28";
	\ar@{} "B28";"F28";
	\ar@{-} "C28";"E28";
	\ar@{} "C28";"F28";
	\ar@{} "C28";"G28";
	\ar@{} "D28";"F28";
	\ar@{} "D28";"G28";
	\ar@{-} "E28";"G28";

	\ar@{} (0,  0);(60,  -122) *\dir<2pt>{*} = "B43";
	\ar@{-} "B43";(58,  -126) *\dir<2pt>{*} = "C43";
	\ar@{-} "C43";(60,  -130) *\dir<2pt>{*} = "D43";
	\ar@{-} "D43";(64,  -130) *\dir<2pt>{*} = "E43";
	\ar@{-} "E43";(66,  -126) *\dir<2pt>{*} = "F43";
	\ar@{-} "F43";(64,  -122) *\dir<2pt>{*} = "G43";
	\ar@{} "B43";"G43";
	\ar@{} "B43";"D43";
	\ar@{-} "B43";"E43";
	\ar@{} "B43";"F43";
	\ar@{-} "C43";"E43";
	\ar@{} "C43";"F43";
	\ar@{} "C43";"G43";
	\ar@{} "D43";"F43";
	\ar@{} "D43";"G43";
	\ar@{-} "E43";"G43";

	\ar@{} (0,  0);(60,  -134) *\dir<2pt>{*} = "B46";
	\ar@{-} "B46";(58,  -138) *\dir<2pt>{*} = "C46";
	\ar@{-} "C46";(60,  -142) *\dir<2pt>{*} = "D46";
	\ar@{-} "D46";(64,  -142) *\dir<2pt>{*} = "E46";
	\ar@{-} "E46";(66,  -138) *\dir<2pt>{*} = "F46";
	\ar@{} "F46";(64,  -134) *\dir<2pt>{*} = "G46";
	\ar@{} "B46";"G46";
	\ar@{} "B46";"D46";
	\ar@{-} "B46";"E46";
	\ar@{} "B46";"F46";
	\ar@{-} "C46";"E46";
	\ar@{} "C46";"F46";
	\ar@{-} "C46";"G46";
	\ar@{} "D46";"F46";
	\ar@{} "D46";"G46";
	\ar@{-} "E46";"G46";

	\ar@{} (0,  0);(60,  -146) *\dir<2pt>{*} = "B65";
	\ar@{-} "B65";(58,  -150) *\dir<2pt>{*} = "C65";
	\ar@{-} "C65";(60,  -154) *\dir<2pt>{*} = "D65";
	\ar@{} "D65";(64,  -154) *\dir<2pt>{*} = "E65";
	\ar@{-} "E65";(66,  -150) *\dir<2pt>{*} = "F65";
	\ar@{-} "F65";(64,  -146) *\dir<2pt>{*} = "G65";
	\ar@{} "B65";"G65";
	\ar@{} "B65";"D65";
	\ar@{} "B65";"E65";
	\ar@{-} "B65";"F65";
	\ar@{-} "C65";"E65";
	\ar@{-} "C65";"F65";
	\ar@{-} "C65";"G65";
	\ar@{-} "D65";"F65";
	\ar@{} "D65";"G65";
	\ar@{} "E65";"G65";

	\ar@{} (0,  0);(76,  -70) *\dir<2pt>{*} = "B98";
	\ar@{-} "B98";(74,  -74) *\dir<2pt>{*} = "C98";
	\ar@{-} "C98";(76,  -78) *\dir<2pt>{*} = "D98";
	\ar@{-} "D98";(80,  -78) *\dir<2pt>{*} = "E98";
	\ar@{-} "E98";(82,  -74) *\dir<2pt>{*} = "F98";
	\ar@{} "F98";(80,  -70) *\dir<2pt>{*} = "G98";
	\ar@{} "B98";"G98";
	\ar@{-} "B98";"D98";
	\ar@{-} "B98";"E98";
	\ar@{} "B98";"F98";
	\ar@{-} "C98";"E98";
	\ar@{} "C98";"F98";
	\ar@{} "C98";"G98";
	\ar@{} "D98";"F98";
	\ar@{} "D98";"G98";
	\ar@{-} "E98";"G98";

	\ar@{} (0,  0);(76,  -86) *\dir<2pt>{*} = "B63";
	\ar@{-} "B63";(74,  -90) *\dir<2pt>{*} = "C63";
	\ar@{-} "C63";(76,  -94) *\dir<2pt>{*} = "D63";
	\ar@{-} "D63";(80,  -94) *\dir<2pt>{*} = "E63";
	\ar@{-} "E63";(82,  -90) *\dir<2pt>{*} = "F63";
	\ar@{-} "F63";(80,  -86) *\dir<2pt>{*} = "G63";
	\ar@{} "B63";"G63";
	\ar@{-} "B63";"D63";
	\ar@{-} "B63";"E63";
	\ar@{} "B63";"F63";
	\ar@{-} "C63";"E63";
	\ar@{} "C63";"F63";
	\ar@{} "C63";"G63";
	\ar@{} "D63";"F63";
	\ar@{} "D63";"G63";
	\ar@{-} "E63";"G63";

	\ar@{} (0,  0);(76,  -98) *\dir<2pt>{*} = "B67";
	\ar@{-} "B67";(74,  -102) *\dir<2pt>{*} = "C67";
	\ar@{-} "C67";(76,  -106) *\dir<2pt>{*} = "D67";
	\ar@{-} "D67";(80,  -106) *\dir<2pt>{*} = "E67";
	\ar@{-} "E67";(82,  -102) *\dir<2pt>{*} = "F67";
	\ar@{} "F67";(80,  -98) *\dir<2pt>{*} = "G67";
	\ar@{-} "B67";"G67";
	\ar@{-} "B67";"D67";
	\ar@{-} "B67";"E67";
	\ar@{} "B67";"F67";
	\ar@{-} "C67";"E67";
	\ar@{} "C67";"F67";
	\ar@{} "C67";"G67";
	\ar@{} "D67";"F67";
	\ar@{} "D67";"G67";
	\ar@{-} "E67";"G67";

	\ar@{} (0,  0);(76,  -110) *\dir<2pt>{*} = "B85";
	\ar@{-} "B85";(74,  -114) *\dir<2pt>{*} = "C85";
	\ar@{-} "C85";(76,  -118) *\dir<2pt>{*} = "D85";
	\ar@{-} "D85";(80,  -118) *\dir<2pt>{*} = "E85";
	\ar@{-} "E85";(82,  -114) *\dir<2pt>{*} = "F85";
	\ar@{-} "F85";(80,  -110) *\dir<2pt>{*} = "G85";
	\ar@{} "B85";"G85";
	\ar@{} "B85";"D85";
	\ar@{} "B85";"E85";
	\ar@{-} "B85";"F85";
	\ar@{-} "C85";"E85";
	\ar@{-} "C85";"F85";
	\ar@{-} "C85";"G85";
	\ar@{-} "D85";"F85";
	\ar@{} "D85";"G85";
	\ar@{} "E85";"G85";

	\ar@{} (0,  0);(76,  -122) *\dir<2pt>{*} = "B86";
	\ar@{-} "B86";(74,  -126) *\dir<2pt>{*} = "C86";
	\ar@{-} "C86";(76,  -130) *\dir<2pt>{*} = "D86";
	\ar@{-} "D86";(80,  -130) *\dir<2pt>{*} = "E86";
	\ar@{-} "E86";(82,  -126) *\dir<2pt>{*} = "F86";
	\ar@{} "F86";(80,  -122) *\dir<2pt>{*} = "G86";
	\ar@{-} "B86";"G86";
	\ar@{-} "B86";"D86";
	\ar@{-} "B86";"E86";
	\ar@{} "B86";"F86";
	\ar@{-} "C86";"E86";
	\ar@{} "C86";"F86";
	\ar@{} "C86";"G86";
	\ar@{} "D86";"F86";
	\ar@{-} "D86";"G86";
	\ar@{-} "E86";"G86";

	\ar@{} (0,  0);(76,  -134) *\dir<2pt>{*} = "B95";
	\ar@{-} "B95";(74,  -138) *\dir<2pt>{*} = "C95";
	\ar@{-} "C95";(76,  -142) *\dir<2pt>{*} = "D95";
	\ar@{-} "D95";(80,  -142) *\dir<2pt>{*} = "E95";
	\ar@{-} "E95";(82,  -138) *\dir<2pt>{*} = "F95";
	\ar@{} "F95";(80,  -134) *\dir<2pt>{*} = "G95";
	\ar@{-} "B95";"G95";
	\ar@{-} "B95";"D95";
	\ar@{-} "B95";"E95";
	\ar@{} "B95";"F95";
	\ar@{-} "C95";"E95";
	\ar@{} "C95";"F95";
	\ar@{-} "C95";"G95";
	\ar@{} "D95";"F95";
	\ar@{-} "D95";"G95";
	\ar@{-} "E95";"G95";

	\ar@{} (0,  0);(76,  -146) *\dir<2pt>{*} = "B96";
	\ar@{-} "B96";(74,  -150) *\dir<2pt>{*} = "C96";
	\ar@{-} "C96";(76,  -154) *\dir<2pt>{*} = "D96";
	\ar@{-} "D96";(80,  -154) *\dir<2pt>{*} = "E96";
	\ar@{-} "E96";(82,  -150) *\dir<2pt>{*} = "F96";
	\ar@{-} "F96";(80,  -146) *\dir<2pt>{*} = "G96";
	\ar@{-} "B96";"G96";
	\ar@{} "B96";"D96";
	\ar@{} "B96";"E96";
	\ar@{-} "B96";"F96";
	\ar@{-} "C96";"E96";
	\ar@{-} "C96";"F96";
	\ar@{-} "C96";"G96";
	\ar@{-} "D96";"F96";
	\ar@{} "D96";"G96";
	\ar@{} "E96";"G96";

	\ar@{} (0,  0);(88,  -70) *\dir<2pt>{*} = "B98";
	\ar@{-} "B98";(86,  -74) *\dir<2pt>{*} = "C98";
	\ar@{-} "C98";(88,  -78) *\dir<2pt>{*} = "D98";
	\ar@{-} "D98";(92,  -78) *\dir<2pt>{*} = "E98";
	\ar@{-} "E98";(94,  -74) *\dir<2pt>{*} = "F98";
	\ar@{-} "F98";(92,  -70) *\dir<2pt>{*} = "G98";
	\ar@{} "B98";"G98";
	\ar@{} "B98";"D98";
	\ar@{-} "B98";"E98";
	\ar@{-} "B98";"F98";
	\ar@{-} "C98";"E98";
	\ar@{-} "C98";"F98";
	\ar@{-} "C98";"G98";
	\ar@{-} "D98";"F98";
	\ar@{} "D98";"G98";
	\ar@{} "E98";"G98";

	\ar@{} (0,  0);(88,  -86) *\dir<2pt>{*} = "B105";
	\ar@{-} "B105";(86,  -90) *\dir<2pt>{*} = "C105";
	\ar@{-} "C105";(88,  -94) *\dir<2pt>{*} = "D105";
	\ar@{-} "D105";(92,  -94) *\dir<2pt>{*} = "E105";
	\ar@{-} "E105";(94,  -90) *\dir<2pt>{*} = "F105";
	\ar@{-} "F105";(92,  -86) *\dir<2pt>{*} = "G105";
	\ar@{-} "B105";"G105";
	\ar@{-} "B105";"D105";
	\ar@{-} "B105";"E105";
	\ar@{-} "B105";"F105";
	\ar@{-} "C105";"E105";
	\ar@{-} "C105";"F105";
	\ar@{} "C105";"G105";
	\ar@{-} "D105";"F105";
	\ar@{} "D105";"G105";
	\ar@{} "E105";"G105";

	\ar@{} (0,  0);(88,  -98) *\dir<2pt>{*} = "B106";
	\ar@{-} "B106";(86,  -102) *\dir<2pt>{*} = "C106";
	\ar@{-} "C106";(88,  -106) *\dir<2pt>{*} = "D106";
	\ar@{-} "D106";(92,  -106) *\dir<2pt>{*} = "E106";
	\ar@{-} "E106";(94,  -102) *\dir<2pt>{*} = "F106";
	\ar@{-} "F106";(92,  -98) *\dir<2pt>{*} = "G106";
	\ar@{-} "B106";"G106";
	\ar@{-} "B106";"D106";
	\ar@{-} "B106";"E106";
	\ar@{-} "B106";"F106";
	\ar@{} "C106";"E106";
	\ar@{-} "C106";"F106";
	\ar@{} "C106";"G106";
	\ar@{-} "D106";"F106";
	\ar@{-} "D106";"G106";
	\ar@{} "E106";"G106";

	\ar@{} (0,  0);(88,  -110) *\dir<2pt>{*} = "B109";
	\ar@{-} "B109";(86,  -114) *\dir<2pt>{*} = "C109";
	\ar@{-} "C109";(88,  -118) *\dir<2pt>{*} = "D109";
	\ar@{-} "D109";(92,  -118) *\dir<2pt>{*} = "E109";
	\ar@{-} "E109";(94,  -114) *\dir<2pt>{*} = "F109";
	\ar@{-} "F109";(92,  -110) *\dir<2pt>{*} = "G109";
	\ar@{-} "B109";"G109";
	\ar@{} "B109";"D109";
	\ar@{-} "B109";"E109";
	\ar@{} "B109";"F109";
	\ar@{-} "C109";"E109";
	\ar@{-} "C109";"F109";
	\ar@{-} "C109";"G109";
	\ar@{-} "D109";"F109";
	\ar@{-} "D109";"G109";
	\ar@{-} "E109";"G109";

	\ar@{} (0,  0);(88,  -122) *\dir<2pt>{*} = "B111";
	\ar@{-} "B111";(86,  -126) *\dir<2pt>{*} = "C111";
	\ar@{-} "C111";(88,  -130) *\dir<2pt>{*} = "D111";
	\ar@{-} "D111";(92,  -130) *\dir<2pt>{*} = "E111";
	\ar@{-} "E111";(94,  -126) *\dir<2pt>{*} = "F111";
	\ar@{-} "F111";(92,  -122) *\dir<2pt>{*} = "G111";
	\ar@{-} "B111";"G111";
	\ar@{} "B111";"D111";
	\ar@{-} "B111";"E111";
	\ar@{-} "B111";"F111";
	\ar@{-} "C111";"E111";
	\ar@{-} "C111";"F111";
	\ar@{-} "C111";"G111";
	\ar@{-} "D111";"F111";
	\ar@{-} "D111";"G111";
	\ar@{-} "E111";"G111";

	\ar@{} (0,  0);(88,  -134) *\dir<2pt>{*} = "B112";
	\ar@{-} "B112";(86,  -138) *\dir<2pt>{*} = "C112";
	\ar@{-} "C112";(88,  -142) *\dir<2pt>{*} = "D112";
	\ar@{-} "D112";(92,  -142) *\dir<2pt>{*} = "E112";
	\ar@{-} "E112";(94,  -138) *\dir<2pt>{*} = "F112";
	\ar@{-} "F112";(92,  -134) *\dir<2pt>{*} = "G112";
	\ar@{-} "B112";"G112";
	\ar@{-} "B112";"D112";
	\ar@{-} "B112";"E112";
	\ar@{-} "B112";"F112";
	\ar@{-} "C112";"E112";
	\ar@{-} "C112";"F112";
	\ar@{-} "C112";"G112";
	\ar@{-} "D112";"F112";
	\ar@{-} "D112";"G112";
	\ar@{-} "E112";"G112";

	\ar@{} (0,  0);(-2,  -86) *\dir<2pt>{*} = "B8";
	\ar@{-} "B8";(-4,  -90) *\dir<2pt>{*} = "C8";
	\ar@{-} "C8";(-2,  -94) *\dir<2pt>{*} = "D8";
	\ar@{} "D8";(2,  -94) *\dir<2pt>{*} = "E8";
	\ar@{} "E8";(4,  -90) *\dir<2pt>{*} = "F8";
	\ar@{-} "F8";(2,  -86) *\dir<2pt>{*} = "G8";
	\ar@{-} "B8";"G8";
	\ar@{-} "B8";"D8";
	\ar@{-} "B8";"E8";
	\ar@{} "B8";"F8";
	\ar@{} "C8";"E8";
	\ar@{} "C8";"F8";
	\ar@{} "C8";"G8";
	\ar@{} "D8";"F8";
	\ar@{} "D8";"G8";
	\ar@{} "E8";"G8";

	\ar@{} (0,  0);(10,  -86) *\dir<2pt>{*} = "B9";
	\ar@{-} "B9";(8,  -90) *\dir<2pt>{*} = "C9";
	\ar@{-} "C9";(10,  -94) *\dir<2pt>{*} = "D9";
	\ar@{} "D9";(14,  -94) *\dir<2pt>{*} = "E9";
	\ar@{} "E9";(16,  -90) *\dir<2pt>{*} = "F9";
	\ar@{-} "F9";(14,  -86) *\dir<2pt>{*} = "G9";
	\ar@{-} "B9";"G9";
	\ar@{-} "B9";"D9";
	\ar@{} "B9";"E9";
	\ar@{} "B9";"F9";
	\ar@{} "C9";"E9";
	\ar@{} "C9";"F9";
	\ar@{} "C9";"G9";
	\ar@{} "D9";"F9";
	\ar@{} "D9";"G9";
	\ar@{-} "E9";"G9";

	\ar@{} (0,  0);(22,  -86) *\dir<2pt>{*} = "B10";
	\ar@{-} "B10";(20,  -90) *\dir<2pt>{*} = "C10";
	\ar@{-} "C10";(22,  -94) *\dir<2pt>{*} = "D10";
	\ar@{-} "D10";(26,  -94) *\dir<2pt>{*} = "E10";
	\ar@{} "E10";(28,  -90) *\dir<2pt>{*} = "F10";
	\ar@{} "F10";(26,  -86) *\dir<2pt>{*} = "G10";
	\ar@{-} "B10";"G10";
	\ar@{-} "B10";"D10";
	\ar@{} "B10";"E10";
	\ar@{-} "B10";"F10";
	\ar@{} "C10";"E10";
	\ar@{} "C10";"F10";
	\ar@{} "C10";"G10";
	\ar@{} "D10";"F10";
	\ar@{} "D10";"G10";
	\ar@{} "E10";"G10";

	\ar@{} (0,  0);(34,  -86) *\dir<2pt>{*} = "B11";
	\ar@{-} "B11";(32,  -90) *\dir<2pt>{*} = "C11";
	\ar@{-} "C11";(34,  -94) *\dir<2pt>{*} = "D11";
	\ar@{-} "D11";(38,  -94) *\dir<2pt>{*} = "E11";
	\ar@{} "E11";(40,  -90) *\dir<2pt>{*} = "F11";
	\ar@{-} "F11";(38,  -86) *\dir<2pt>{*} = "G11";
	\ar@{-} "B11";"G11";
	\ar@{-} "B11";"D11";
	\ar@{} "B11";"E11";
	\ar@{} "B11";"F11";
	\ar@{} "C11";"E11";
	\ar@{} "C11";"F11";
	\ar@{} "C11";"G11";
	\ar@{} "D11";"F11";
	\ar@{} "D11";"G11";
	\ar@{} "E11";"G11";

	\ar@{} (0,  0);(46,  -86) *\dir<2pt>{*} = "B12";
	\ar@{-} "B12";(44,  -90) *\dir<2pt>{*} = "C12";
	\ar@{-} "C12";(46,  -94) *\dir<2pt>{*} = "D12";
	\ar@{-} "D12";(50,  -94) *\dir<2pt>{*} = "E12";
	\ar@{-} "E12";(52,  -90) *\dir<2pt>{*} = "F12";
	\ar@{-} "F12";(50,  -86) *\dir<2pt>{*} = "G12";
	\ar@{} "B12";"G12";
	\ar@{-} "B12";"D12";
	\ar@{} "B12";"E12";
	\ar@{} "B12";"F12";
	\ar@{} "C12";"E12";
	\ar@{} "C12";"F12";
	\ar@{} "C12";"G12";
	\ar@{} "D12";"F12";
	\ar@{} "D12";"G12";
	\ar@{} "E12";"G12";

	\ar@{} (0,  0);(-2,  -98) *\dir<2pt>{*} = "B21";
	\ar@{-} "B21";(-4,  -102) *\dir<2pt>{*} = "C21";
	\ar@{-} "C21";(-2,  -106) *\dir<2pt>{*} = "D21";
	\ar@{-} "D21";(2,  -106) *\dir<2pt>{*} = "E21";
	\ar@{-} "E21";(4,  -102) *\dir<2pt>{*} = "F21";
	\ar@{} "F21";(2,  -98) *\dir<2pt>{*} = "G21";
	\ar@{} "B21";"G21";
	\ar@{-} "B21";"D21";
	\ar@{} "B21";"E21";
	\ar@{} "B21";"F21";
	\ar@{} "C21";"E21";
	\ar@{} "C21";"F21";
	\ar@{} "C21";"G21";
	\ar@{} "D21";"F21";
	\ar@{-} "D21";"G21";
	\ar@{-} "E21";"G21";

	\ar@{} (0,  0);(10,  -98) *\dir<2pt>{*} = "B22";
	\ar@{-} "B22";(8,  -102) *\dir<2pt>{*} = "C22";
	\ar@{-} "C22";(10,  -106) *\dir<2pt>{*} = "D22";
	\ar@{-} "D22";(14,  -106) *\dir<2pt>{*} = "E22";
	\ar@{} "E22";(16,  -102) *\dir<2pt>{*} = "F22";
	\ar@{-} "F22";(14,  -98) *\dir<2pt>{*} = "G22";
	\ar@{} "B22";"G22";
	\ar@{-} "B22";"D22";
	\ar@{-} "B22";"E22";
	\ar@{} "B22";"F22";
	\ar@{} "C22";"E22";
	\ar@{} "C22";"F22";
	\ar@{} "C22";"G22";
	\ar@{} "D22";"F22";
	\ar@{} "D22";"G22";
	\ar@{-} "E22";"G22";

	\ar@{} (0,  0);(22,  -98) *\dir<2pt>{*} = "B23";
	\ar@{-} "B23";(20,  -102) *\dir<2pt>{*} = "C23";
	\ar@{} "C23";(22,  -106) *\dir<2pt>{*} = "D23";
	\ar@{-} "D23";(26,  -106) *\dir<2pt>{*} = "E23";
	\ar@{-} "E23";(28,  -102) *\dir<2pt>{*} = "F23";
	\ar@{} "F23";(26,  -98) *\dir<2pt>{*} = "G23";
	\ar@{-} "B23";"G23";
	\ar@{-} "B23";"D23";
	\ar@{} "B23";"E23";
	\ar@{} "B23";"F23";
	\ar@{} "C23";"E23";
	\ar@{} "C23";"F23";
	\ar@{} "C23";"G23";
	\ar@{} "D23";"F23";
	\ar@{-} "D23";"G23";
	\ar@{-} "E23";"G23";

	\ar@{} (0,  0);(34,  -98) *\dir<2pt>{*} = "B25";
	\ar@{-} "B25";(32,  -102) *\dir<2pt>{*} = "C25";
	\ar@{} "C25";(34,  -106) *\dir<2pt>{*} = "D25";
	\ar@{-} "D25";(38,  -106) *\dir<2pt>{*} = "E25";
	\ar@{} "E25";(40,  -102) *\dir<2pt>{*} = "F25";
	\ar@{-} "F25";(38,  -98) *\dir<2pt>{*} = "G25";
	\ar@{-} "B25";"G25";
	\ar@{-} "B25";"D25";
	\ar@{} "B25";"E25";
	\ar@{} "B25";"F25";
	\ar@{} "C25";"E25";
	\ar@{} "C25";"F25";
	\ar@{} "C25";"G25";
	\ar@{} "D25";"F25";
	\ar@{-} "D25";"G25";
	\ar@{-} "E25";"G25";

	\ar@{} (0,  0);(46,  -98) *\dir<2pt>{*} = "B26";
	\ar@{-} "B26";(44,  -102) *\dir<2pt>{*} = "C26";
	\ar@{-} "C26";(46,  -106) *\dir<2pt>{*} = "D26";
	\ar@{-} "D26";(50,  -106) *\dir<2pt>{*} = "E26";
	\ar@{} "E26";(52,  -102) *\dir<2pt>{*} = "F26";
	\ar@{-} "F26";(50,  -98) *\dir<2pt>{*} = "G26";
	\ar@{-} "B26";"G26";
	\ar@{-} "B26";"D26";
	\ar@{-} "B26";"E26";
	\ar@{} "B26";"F26";
	\ar@{} "C26";"E26";
	\ar@{} "C26";"F26";
	\ar@{} "C26";"G26";
	\ar@{} "D26";"F26";
	\ar@{} "D26";"G26";
	\ar@{} "E26";"G26";

	\ar@{} (0,  0);(-2,  -110) *\dir<2pt>{*} = "B29";
	\ar@{} "B29";(-4,  -114) *\dir<2pt>{*} = "C29";
	\ar@{-} "C29";(-2,  -118) *\dir<2pt>{*} = "D29";
	\ar@{-} "D29";(2,  -118) *\dir<2pt>{*} = "E29";
	\ar@{} "E29";(4,  -114) *\dir<2pt>{*} = "F29";
	\ar@{-} "F29";(2,  -110) *\dir<2pt>{*} = "G29";
	\ar@{-} "B29";"G29";
	\ar@{-} "B29";"D29";
	\ar@{} "B29";"E29";
	\ar@{} "B29";"F29";
	\ar@{} "C29";"E29";
	\ar@{} "C29";"F29";
	\ar@{} "C29";"G29";
	\ar@{} "D29";"F29";
	\ar@{-} "D29";"G29";
	\ar@{-} "E29";"G29";

	\ar@{} (0,  0);(10,  -110) *\dir<2pt>{*} = "B30";
	\ar@{-} "B30";(8,  -114) *\dir<2pt>{*} = "C30";
	\ar@{-} "C30";(10,  -118) *\dir<2pt>{*} = "D30";
	\ar@{} "D30";(14,  -118) *\dir<2pt>{*} = "E30";
	\ar@{} "E30";(16,  -114) *\dir<2pt>{*} = "F30";
	\ar@{-} "F30";(14,  -110) *\dir<2pt>{*} = "G30";
	\ar@{-} "B30";"G30";
	\ar@{-} "B30";"D30";
	\ar@{} "B30";"E30";
	\ar@{} "B30";"F30";
	\ar@{} "C30";"E30";
	\ar@{} "C30";"F30";
	\ar@{} "C30";"G30";
	\ar@{} "D30";"F30";
	\ar@{-} "D30";"G30";
	\ar@{-} "E30";"G30";

	\ar@{} (0,  0);(22,  -110) *\dir<2pt>{*} = "B32";
	\ar@{-} "B32";(20,  -114) *\dir<2pt>{*} = "C32";
	\ar@{-} "C32";(22,  -118) *\dir<2pt>{*} = "D32";
	\ar@{-} "D32";(26,  -118) *\dir<2pt>{*} = "E32";
	\ar@{-} "E32";(28,  -114) *\dir<2pt>{*} = "F32";
	\ar@{-} "F32";(26,  -110) *\dir<2pt>{*} = "G32";
	\ar@{} "B32";"G32";
	\ar@{} "B32";"D32";
	\ar@{-} "B32";"E32";
	\ar@{} "B32";"F32";
	\ar@{} "C32";"E32";
	\ar@{} "C32";"F32";
	\ar@{} "C32";"G32";
	\ar@{} "D32";"F32";
	\ar@{} "D32";"G32";
	\ar@{-} "E32";"G32";

	\ar@{} (0,    0);(34,  -110) *\dir<2pt>{*} = "B34";
	\ar@{-} "B34";(32,  -114) *\dir<2pt>{*} = "C34";
	\ar@{-} "C34";(34,  -118) *\dir<2pt>{*} = "D34";
	\ar@{-} "D34";(38,  -118) *\dir<2pt>{*} = "E34";
	\ar@{-} "E34";(40,  -114) *\dir<2pt>{*} = "F34";
	\ar@{} "F34";(38,  -110) *\dir<2pt>{*} = "G34";
	\ar@{-} "B34";"G34";
	\ar@{-} "B34";"D34";
	\ar@{} "B34";"E34";
	\ar@{} "B34";"F34";
	\ar@{} "C34";"E34";
	\ar@{} "C34";"F34";
	\ar@{} "C34";"G34";
	\ar@{} "D34";"F34";
	\ar@{} "D34";"G34";
	\ar@{-} "E34";"G34";

	\ar@{} (0,    0);(46,  -110) *\dir<2pt>{*} = "B37";
	\ar@{-} "B37";(44,  -114) *\dir<2pt>{*} = "C37";
	\ar@{-} "C37";(46,  -118) *\dir<2pt>{*} = "D37";
	\ar@{-} "D37";(50,  -118) *\dir<2pt>{*} = "E37";
	\ar@{-} "E37";(52,  -114) *\dir<2pt>{*} = "F37";
	\ar@{} "F37" ;(50,  -110) *\dir<2pt>{*} = "G37";
	\ar@{-} "B37";"G37";
	\ar@{} "B37";"D37";
	\ar@{-} "B37";"E37";
	\ar@{} "B37";"F37";
	\ar@{} "C37";"E37";
	\ar@{} "C37";"F37";
	\ar@{} "C37";"G37";
	\ar@{} "D37";"F37";
	\ar@{} "D37";"G37";
	\ar@{-} "E37";"G37";

	\ar@{} (0,  0);(-2,  -122) *\dir<2pt>{*} = "B42";
	\ar@{-} "B42";(-4,  -126) *\dir<2pt>{*} = "C42";
	\ar@{-} "C42";(-2,  -130) *\dir<2pt>{*} = "D42";
	\ar@{-} "D42";(2,  -130) *\dir<2pt>{*} = "E42";
	\ar@{} "E42";(4,  -126) *\dir<2pt>{*} = "F42";
	\ar@{-} "F42";(2,  -122) *\dir<2pt>{*} = "G42";
	\ar@{-} "B42";"G42";
	\ar@{-} "B42";"D42";
	\ar@{-} "B42";"E42";
	\ar@{} "B42";"F42";
	\ar@{} "C42";"E42";
	\ar@{} "C42";"F42";
	\ar@{} "C42";"G42";
	\ar@{} "D42";"F42";
	\ar@{} "D42";"G42";
	\ar@{-} "E42";"G42";

	\ar@{} (0,  0);(10,  -122) *\dir<2pt>{*} = "B44";
	\ar@{-} "B44";(8,  -126) *\dir<2pt>{*} = "C44";
	\ar@{-} "C44";(10,  -130) *\dir<2pt>{*} = "D44";
	\ar@{-} "D44";(14,  -130) *\dir<2pt>{*} = "E44";
	\ar@{-} "E44";(16,  -126) *\dir<2pt>{*} = "F44";
	\ar@{} "F44";(14,  -122) *\dir<2pt>{*} = "G44";
	\ar@{-} "B44";"G44";
	\ar@{-} "B44";"D44";
	\ar@{-} "B44";"E44";
	\ar@{} "B44";"F44";
	\ar@{} "C44";"E44";
	\ar@{} "C44";"F44";
	\ar@{} "C44";"G44";
	\ar@{} "D44";"F44";
	\ar@{} "D44";"G44";
	\ar@{-} "E44";"G44";

	\ar@{} (0,  0);(22,  -122) *\dir<2pt>{*} = "B45";
	\ar@{-} "B45";(20,  -126) *\dir<2pt>{*} = "C45";
	\ar@{-} "C45";(22,  -130) *\dir<2pt>{*} = "D45";
	\ar@{-} "D45";(26,  -130) *\dir<2pt>{*} = "E45";
	\ar@{-} "E45";(28,  -126) *\dir<2pt>{*} = "F45";
	\ar@{} "F45";(26,  -122) *\dir<2pt>{*} = "G45";
	\ar@{-} "B45";"G45";
	\ar@{} "B45";"D45";
	\ar@{-} "B45";"E45";
	\ar@{} "B45";"F45";
	\ar@{-} "C45";"E45";
	\ar@{} "C45";"F45";
	\ar@{} "C45";"G45";
	\ar@{} "D45";"F45";
	\ar@{} "D45";"G45";
	\ar@{-} "E45";"G45";

	\ar@{} (0,  0);(34,  -122) *\dir<2pt>{*} = "B48";
	\ar@{-} "B48";(32,  -126) *\dir<2pt>{*} = "C48";
	\ar@{-} "C48";(34,  -130) *\dir<2pt>{*} = "D48";
	\ar@{-} "D48";(38,  -130) *\dir<2pt>{*} = "E48";
	\ar@{-} "E48";(40,  -126) *\dir<2pt>{*} = "F48";
	\ar@{} "F48";(38,  -122) *\dir<2pt>{*} = "G48";
	\ar@{-} "B48";"G48";
	\ar@{-} "B48";"D48";
	\ar@{-} "B48";"E48";
	\ar@{} "B48";"F48";
	\ar@{} "C48";"E48";
	\ar@{} "C48";"F48";
	\ar@{} "C48";"G48";
	\ar@{} "D48";"F48";
	\ar@{-} "D48";"G48";
	\ar@{} "E48";"G48";

	\ar@{} (0,  0);(46,  -122) *\dir<2pt>{*} = "B50";
	\ar@{-} "B50";(44,  -126) *\dir<2pt>{*} = "C50";
	\ar@{-} "C50";(46,  -130) *\dir<2pt>{*} = "D50";
	\ar@{-} "D50";(50,  -130) *\dir<2pt>{*} = "E50";
	\ar@{-} "E50";(52,  -126) *\dir<2pt>{*} = "F50";
	\ar@{-} "F50";(50,  -122) *\dir<2pt>{*} = "G50";
	\ar@{-} "B50";"G50";
	\ar@{-} "B50";"D50";
	\ar@{-} "B50";"E50";
	\ar@{} "B50";"F50";
	\ar@{} "C50";"E50";
	\ar@{} "C50";"F50";
	\ar@{} "C50";"G50";
	\ar@{} "D50";"F50";
	\ar@{} "D50";"G50";
	\ar@{} "E50";"G50";

	\ar@{} (0,  0);(-2,  -134) *\dir<2pt>{*} = "B51";
	\ar@{-} "B51";(-4,  -138) *\dir<2pt>{*} = "C51";
	\ar@{-} "C51";(-2,  -142) *\dir<2pt>{*} = "D51";
	\ar@{-} "D51";(2,  -142) *\dir<2pt>{*} = "E51";
	\ar@{-} "E51";(4,  -138) *\dir<2pt>{*} = "F51";
	\ar@{} "F51";(2,  -134) *\dir<2pt>{*} = "G51";
	\ar@{-} "B51";"G51";
	\ar@{-} "B51";"D51";
	\ar@{} "B51";"E51";
	\ar@{} "B51";"F51";
	\ar@{-} "C51";"E51";
	\ar@{} "C51";"F51";
	\ar@{} "C51";"G51";
	\ar@{} "D51";"F51";
	\ar@{} "D51";"G51";
	\ar@{-} "E51";"G51";

	\ar@{} (0,  0);(10,  -134) *\dir<2pt>{*} = "B52";
	\ar@{-} "B52";(8,  -138) *\dir<2pt>{*} = "C52";
	\ar@{-} "C52";(10,  -142) *\dir<2pt>{*} = "D52";
	\ar@{-} "D52";(14,  -142) *\dir<2pt>{*} = "E52";
	\ar@{} "E52";(16,  -138) *\dir<2pt>{*} = "F52";
	\ar@{-} "F52";(14,  -134) *\dir<2pt>{*} = "G52";
	\ar@{-} "B52";"G52";
	\ar@{} "B52";"D52";
	\ar@{-} "B52";"E52";
	\ar@{} "B52";"F52";
	\ar@{-} "C52";"E52";
	\ar@{} "C52";"F52";
	\ar@{} "C52";"G52";
	\ar@{} "D52";"F52";
	\ar@{-} "D52";"G52";
	\ar@{} "E52";"G52";

	\ar@{} (0,  0);(22,  -134) *\dir<2pt>{*} = "B53";
	\ar@{-} "B53";(20,  -138) *\dir<2pt>{*} = "C53";
	\ar@{-} "C53";(22,  -142) *\dir<2pt>{*} = "D53";
	\ar@{-} "D53";(26,  -142) *\dir<2pt>{*} = "E53";
	\ar@{-} "E53";(28,  -138) *\dir<2pt>{*} = "F53";
	\ar@{-} "F53";(26,  -134) *\dir<2pt>{*} = "G53";
	\ar@{} "B53";"G53";
	\ar@{} "B53";"D53";
	\ar@{} "B53";"E53";
	\ar@{-} "B53";"F53";
	\ar@{} "C53";"E53";
	\ar@{} "C53";"F53";
	\ar@{-} "C53";"G53";
	\ar@{-} "D53";"F53";
	\ar@{} "D53";"G53";
	\ar@{} "E53";"G53";

	\ar@{} (0,  0);(34,  -134) *\dir<2pt>{*} = "B55";
	\ar@{-} "B55";(32,  -138) *\dir<2pt>{*} = "C55";
	\ar@{-} "C55";(34,  -142) *\dir<2pt>{*} = "D55";
	\ar@{-} "D55";(38,  -142) *\dir<2pt>{*} = "E55";
	\ar@{-} "E55";(40,  -138) *\dir<2pt>{*} = "F55";
	\ar@{} "F55";(38,  -134) *\dir<2pt>{*} = "G55";
	\ar@{-} "B55";"G55";
	\ar@{} "B55";"D55";
	\ar@{-} "B55";"E55";
	\ar@{} "B55";"F55";
	\ar@{} "C55";"E55";
	\ar@{} "C55";"F55";
	\ar@{} "C55";"G55";
	\ar@{} "D55";"F55";
	\ar@{-} "D55";"G55";
	\ar@{-} "E55";"G55";

	\ar@{} (0,  0);(46,  -134) *\dir<2pt>{*} = "B56";
	\ar@{-} "B56";(44,  -138) *\dir<2pt>{*} = "C56";
	\ar@{-} "C56";(46,  -142) *\dir<2pt>{*} = "D56";
	\ar@{-} "D56";(50,  -142) *\dir<2pt>{*} = "E56";
	\ar@{-} "E56";(52,  -138) *\dir<2pt>{*} = "F56";
	\ar@{-} "F56";(50,  -134) *\dir<2pt>{*} = "G56";
	\ar@{} "B56";"G56";
	\ar@{} "B56";"D56";
	\ar@{} "B56";"E56";
	\ar@{-} "B56";"F56";
	\ar@{} "C56";"E56";
	\ar@{-} "C56";"F56";
	\ar@{-} "C56";"G56";
	\ar@{} "D56";"F56";
	\ar@{} "D56";"G56";
	\ar@{} "E56";"G56";

	\ar@{} (0,  0);(-2,  -146) *\dir<2pt>{*} = "B59";
	\ar@{-} "B59";(-4,  -150) *\dir<2pt>{*} = "C59";
	\ar@{-} "C59";(-2,  -154) *\dir<2pt>{*} = "D59";
	\ar@{-} "D59";(2,  -154) *\dir<2pt>{*} = "E59";
	\ar@{-} "E59";(4,  -150) *\dir<2pt>{*} = "F59";
	\ar@{-} "F59";(2,  -146) *\dir<2pt>{*} = "G59";
	\ar@{-} "B59";"G59";
	\ar@{-} "B59";"D59";
	\ar@{} "B59";"E59";
	\ar@{-} "B59";"F59";
	\ar@{} "C59";"E59";
	\ar@{} "C59";"F59";
	\ar@{} "C59";"G59";
	\ar@{} "D59";"F59";
	\ar@{} "D59";"G59";
	\ar@{} "E59";"G59";

	\ar@{} (0,  0);(10,  -146) *\dir<2pt>{*} = "B66";
	\ar@{-} "B66";(8,  -150) *\dir<2pt>{*} = "C66";
	\ar@{-} "C66";(10,  -154) *\dir<2pt>{*} = "D66";
	\ar@{-} "D66";(14,  -154) *\dir<2pt>{*} = "E66";
	\ar@{-} "E66";(16,  -150) *\dir<2pt>{*} = "F66";
	\ar@{-} "F66";(14,  -146) *\dir<2pt>{*} = "G66";
	\ar@{-} "B66";"G66";
	\ar@{-} "B66";"D66";
	\ar@{-} "B66";"E66";
	\ar@{-} "B66";"F66";
	\ar@{} "C66";"E66";
	\ar@{} "C66";"F66";
	\ar@{} "C66";"G66";
	\ar@{} "D66";"F66";
	\ar@{} "D66";"G66";
	\ar@{} "E66";"G66";

	\ar@{} (0,  0);(22,  -146) *\dir<2pt>{*} = "B68";
	\ar@{-} "B68";(20,  -150) *\dir<2pt>{*} = "C68";
	\ar@{-} "C68";(22,  -154) *\dir<2pt>{*} = "D68";
	\ar@{-} "D68";(26,  -154) *\dir<2pt>{*} = "E68";
	\ar@{-} "E68";(28,  -150) *\dir<2pt>{*} = "F68";
	\ar@{-} "F68";(26,  -146) *\dir<2pt>{*} = "G68";
	\ar@{} "B68";"G68";
	\ar@{} "B68";"D68";
	\ar@{} "B68";"E68";
	\ar@{-} "B68";"F68";
	\ar@{} "C68";"E68";
	\ar@{-} "C68";"F68";
	\ar@{-} "C68";"G68";
	\ar@{-} "D68";"F68";
	\ar@{} "D68";"G68";
	\ar@{} "E68";"G68";

	\ar@{} (0,  0);(34,  -146) *\dir<2pt>{*} = "B72";
	\ar@{-} "B72";(32,  -150) *\dir<2pt>{*} = "C72";
	\ar@{-} "C72";(34,  -154) *\dir<2pt>{*} = "D72";
	\ar@{-} "D72";(38,  -154) *\dir<2pt>{*} = "E72";
	\ar@{-} "E72";(40,  -150) *\dir<2pt>{*} = "F72";
	\ar@{-} "F72";(38,  -146) *\dir<2pt>{*} = "G72";
	\ar@{-} "B72";"G72";
	\ar@{-} "B72";"D72";
	\ar@{} "B72";"E72";
	\ar@{} "B72";"F72";
	\ar@{} "C72";"E72";
	\ar@{-} "C72";"F72";
	\ar@{} "C72";"G72";
	\ar@{-} "D72";"F72";
	\ar@{} "D72";"G72";
	\ar@{} "E72";"G72";

	\ar@{} (0,  0);(46,  -146) *\dir<2pt>{*} = "B73";
	\ar@{-} "B73";(44,  -150) *\dir<2pt>{*} = "C73";
	\ar@{-} "C73";(46,  -154) *\dir<2pt>{*} = "D73";
	\ar@{-} "D73";(50,  -154) *\dir<2pt>{*} = "E73";
	\ar@{-} "E73";(52,  -150) *\dir<2pt>{*} = "F73";
	\ar@{-} "F73";(50,  -146) *\dir<2pt>{*} = "G73";
	\ar@{} "B73";"G73";
	\ar@{} "B73";"D73";
	\ar@{} "B73";"E73";
	\ar@{-} "B73";"F73";
	\ar@{-} "C73";"E73";
	\ar@{} "C73";"F73";
	\ar@{-} "C73";"G73";
	\ar@{-} "D73";"F73";
	\ar@{} "D73";"G73";
	\ar@{} "E73";"G73";

	\ar@{} (0,  0);(-2,  -158) *\dir<2pt>{*} = "B74";
	\ar@{-} "B74";(-4,  -162) *\dir<2pt>{*} = "C74";
	\ar@{-} "C74";(-2,  -166) *\dir<2pt>{*} = "D74";
	\ar@{-} "D74";(2,  -166) *\dir<2pt>{*} = "E74";
	\ar@{-} "E74";(4,  -162) *\dir<2pt>{*} = "F74";
	\ar@{} "F74";(2,  -158) *\dir<2pt>{*} = "G74";
	\ar@{-} "B74";"G74";
	\ar@{-} "B74";"D74";
	\ar@{} "B74";"E74";
	\ar@{} "B74";"F74";
	\ar@{-} "C74";"E74";
	\ar@{} "C74";"F74";
	\ar@{} "C74";"G74";
	\ar@{} "D74";"F74";
	\ar@{-} "D74";"G74";
	\ar@{-} "E74";"G74";

	\ar@{} (0,  0);(10,  -158) *\dir<2pt>{*} = "B75";
	\ar@{-} "B75";(8,  -162) *\dir<2pt>{*} = "C75";
	\ar@{-} "C75";(10,  -166) *\dir<2pt>{*} = "D75";
	\ar@{-} "D75";(14,  -166) *\dir<2pt>{*} = "E75";
	\ar@{-} "E75";(16,  -162) *\dir<2pt>{*} = "F75";
	\ar@{-} "F75";(14,  -158) *\dir<2pt>{*} = "G75";
	\ar@{-} "B75";"G75";
	\ar@{} "B75";"D75";
	\ar@{-} "B75";"E75";
	\ar@{} "B75";"F75";
	\ar@{} "C75";"E75";
	\ar@{} "C75";"F75";
	\ar@{} "C75";"G75";
	\ar@{} "D75";"F75";
	\ar@{-} "D75";"G75";
	\ar@{-} "E75";"G75";

	\ar@{} (0,  0);(22,  -158) *\dir<2pt>{*} = "B76";
	\ar@{-} "B76";(20,  -162) *\dir<2pt>{*} = "C76";
	\ar@{-} "C76";(22,  -166) *\dir<2pt>{*} = "D76";
	\ar@{-} "D76";(26,  -166) *\dir<2pt>{*} = "E76";
	\ar@{-} "E76";(28,  -162) *\dir<2pt>{*} = "F76";
	\ar@{} "F76";(26,  -158) *\dir<2pt>{*} = "G76";
	\ar@{-} "B76";"G76";
	\ar@{} "B76";"D76";
	\ar@{-} "B76";"E76";
	\ar@{} "B76";"F76";
	\ar@{-} "C76";"E76";
	\ar@{} "C76";"F76";
	\ar@{} "C76";"G76";
	\ar@{} "D76";"F76";
	\ar@{-} "D76";"G76";
	\ar@{-} "E76";"G76";

	\ar@{} (0,  0);(34,  -158) *\dir<2pt>{*} = "B78";
	\ar@{-} "B78";(32,  -162) *\dir<2pt>{*} = "C78";
	\ar@{-} "C78";(34,  -166) *\dir<2pt>{*} = "D78";
	\ar@{-} "D78";(38,  -166) *\dir<2pt>{*} = "E78";
	\ar@{-} "E78";(40,  -162) *\dir<2pt>{*} = "F78";
	\ar@{-} "F78";(38,  -158) *\dir<2pt>{*} = "G78";
	\ar@{-} "B78";"G78";
	\ar@{-} "B78";"D78";
	\ar@{-} "B78";"E78";
	\ar@{} "B78";"F78";
	\ar@{} "C78";"E78";
	\ar@{-} "C78";"F78";
	\ar@{} "C78";"G78";
	\ar@{} "D78";"F78";
	\ar@{} "D78";"G78";
	\ar@{} "E78";"G78";

	\ar@{} (0,  0);(46,  -158) *\dir<2pt>{*} = "B89";
	\ar@{-} "B89";(44,  -162) *\dir<2pt>{*} = "C89";
	\ar@{-} "C89";(46,  -166) *\dir<2pt>{*} = "D89";
	\ar@{-} "D89";(50,  -166) *\dir<2pt>{*} = "E89";
	\ar@{-} "E89";(52,  -162) *\dir<2pt>{*} = "F89";
	\ar@{-} "F89";(50,  -158) *\dir<2pt>{*} = "G89";
	\ar@{-} "B89";"G89";
	\ar@{} "B89";"D89";
	\ar@{-} "B89";"E89";
	\ar@{} "B89";"F89";
	\ar@{-} "C89";"E89";
	\ar@{} "C89";"F89";
	\ar@{} "C89";"G89";
	\ar@{} "D89";"F89";
	\ar@{-} "D89";"G89";
	\ar@{-} "E89";"G89";

	\ar@{} (0,  0);(-2,  -170) *\dir<2pt>{*} = "B91";
	\ar@{-} "B91";(-4,  -174) *\dir<2pt>{*} = "C91";
	\ar@{-} "C91";(-2,  -178) *\dir<2pt>{*} = "D91";
	\ar@{-} "D91";(2,  -178) *\dir<2pt>{*} = "E91";
	\ar@{-} "E91";(4,  -174) *\dir<2pt>{*} = "F91";
	\ar@{-} "F91";(2,  -170) *\dir<2pt>{*} = "G91";
	\ar@{-} "B91";"G91";
	\ar@{} "B91";"D91";
	\ar@{-} "B91";"E91";
	\ar@{} "B91";"F91";
	\ar@{} "C91";"E91";
	\ar@{-} "C91";"F91";
	\ar@{-} "C91";"G91";
	\ar@{} "D91";"F91";
	\ar@{} "D91";"G91";
	\ar@{-} "E91";"G91";

	\ar@{} (0,  0);(10,  -170) *\dir<2pt>{*} = "B93";
	\ar@{-} "B93";(8,  -174) *\dir<2pt>{*} = "C93";
	\ar@{-} "C93";(10,  -178) *\dir<2pt>{*} = "D93";
	\ar@{-} "D93";(14,  -178) *\dir<2pt>{*} = "E93";
	\ar@{-} "E93";(16,  -174) *\dir<2pt>{*} = "F93";
	\ar@{-} "F93";(14,  -170) *\dir<2pt>{*} = "G93";
	\ar@{-} "B93";"G93";
	\ar@{} "B93";"D93";
	\ar@{-} "B93";"E93";
	\ar@{} "B93";"F93";
	\ar@{} "C93";"E93";
	\ar@{-} "C93";"F93";
	\ar@{} "C93";"G93";
	\ar@{} "D93";"F93";
	\ar@{-} "D93";"G93";
	\ar@{-} "E93";"G93";

	\ar@{} (0,  0);(22,  -170) *\dir<2pt>{*} = "B99";
	\ar@{-} "B99";(20,  -174) *\dir<2pt>{*} = "C99";
	\ar@{-} "C99";(22,  -178) *\dir<2pt>{*} = "D99";
	\ar@{-} "D99";(26,  -178) *\dir<2pt>{*} = "E99";
	\ar@{-} "E99";(28,  -174) *\dir<2pt>{*} = "F99";
	\ar@{-} "F99";(26,  -170) *\dir<2pt>{*} = "G99";
	\ar@{-} "B99";"G99";
	\ar@{} "B99";"D99";
	\ar@{-} "B99";"E99";
	\ar@{} "B99";"F99";
	\ar@{-} "C99";"E99";
	\ar@{-} "C99";"F99";
	\ar@{-} "C99";"G99";
	\ar@{} "D99";"F99";
	\ar@{-} "D99";"G99";
	\ar@{} "E99";"G99";

	\ar@{} (0,  0);(-2,  -190) *\dir<2pt>{*} = "B87";
	\ar@{-} "B87";(-4,  -194) *\dir<2pt>{*} = "C87";
	\ar@{-} "C87";(-2,  -198) *\dir<2pt>{*} = "D87";
	\ar@{-} "D87";(2,  -198) *\dir<2pt>{*} = "E87";
	\ar@{-} "E87";(4,  -194) *\dir<2pt>{*} = "F87";
	\ar@{-} "F87";(2,  -190) *\dir<2pt>{*} = "G87";
	\ar@{-} "B87";"G87";
	\ar@{} "B87";"D87";
	\ar@{} "B87";"E87";
	\ar@{-} "B87";"F87";
	\ar@{-} "C87";"E87";
	\ar@{} "C87";"F87";
	\ar@{-} "C87";"G87";
	\ar@{} "D87";"F87";
	\ar@{} "D87";"G87";
	\ar@{-} "E87";"G87";

	\ar@{} (0,  0);(10,  -190) *\dir<2pt>{*} = "B100";
	\ar@{-} "B100";(8,  -194) *\dir<2pt>{*} = "C100";
	\ar@{-} "C100";(10,  -198) *\dir<2pt>{*} = "D100";
	\ar@{-} "D100";(14,  -198) *\dir<2pt>{*} = "E100";
	\ar@{-} "E100";(16,  -194) *\dir<2pt>{*} = "F100";
	\ar@{-} "F100";(14,  -190) *\dir<2pt>{*} = "G100";
	\ar@{-} "B100";"G100";
	\ar@{} "B100";"D100";
	\ar@{} "B100";"E100";
	\ar@{-} "B100";"F100";
	\ar@{-} "C100";"E100";
	\ar@{} "C100";"F100";
	\ar@{-} "C100";"G100";
	\ar@{-} "D100";"F100";
	\ar@{-} "D100";"G100";
	\ar@{} "E100";"G100";

	\ar@{} (0,  0);(22,  -190) *\dir<2pt>{*} = "B108";
	\ar@{-} "B108";(20,  -194) *\dir<2pt>{*} = "C108";
	\ar@{-} "C108";(22,  -198) *\dir<2pt>{*} = "D108";
	\ar@{-} "D108";(26,  -198) *\dir<2pt>{*} = "E108";
	\ar@{-} "E108";(28,  -194) *\dir<2pt>{*} = "F108";
	\ar@{-} "F108";(26,  -190) *\dir<2pt>{*} = "G108";
	\ar@{-} "B108";"G108";
	\ar@{-} "B108";"D108";
	\ar@{} "B108";"E108";
	\ar@{-} "B108";"F108";
	\ar@{-} "C108";"E108";
	\ar@{} "C108";"F108";
	\ar@{-} "C108";"G108";
	\ar@{-} "D108";"F108";
	\ar@{} "D108";"G108";
	\ar@{-} "E108";"G108";

	\ar@{} (0,  0);(-2,  -206) *\dir<2pt>{*} = "B60";
	\ar@{-} "B60";(-4,  -210) *\dir<2pt>{*} = "C60";
	\ar@{-} "C60";(-2,  -214) *\dir<2pt>{*} = "D60";
	\ar@{-} "D60";(2,  -214) *\dir<2pt>{*} = "E60";
	\ar@{-} "E60";(4,  -210) *\dir<2pt>{*} = "F60";
	\ar@{-} "F60";(2,  -206) *\dir<2pt>{*} = "G60";
	\ar@{-} "B60";"G60";
	\ar@{} "B60";"D60";
	\ar@{} "B60";"E60";
	\ar@{-} "B60";"F60";
	\ar@{} "C60";"E60";
	\ar@{} "C60";"F60";
	\ar@{-} "C60";"G60";
	\ar@{} "D60";"F60";
	\ar@{} "D60";"G60";
	\ar@{} "E60";"G60";

	\ar@{} (0,  0);(10,  -206) *\dir<2pt>{*} = "B70";
	\ar@{-} "B70";(8,  -210) *\dir<2pt>{*} = "C70";
	\ar@{-} "C70";(10,  -214) *\dir<2pt>{*} = "D70";
	\ar@{-} "D70";(14,  -214) *\dir<2pt>{*} = "E70";
	\ar@{-} "E70";(16,  -210) *\dir<2pt>{*} = "F70";
	\ar@{-} "F70";(14,  -206) *\dir<2pt>{*} = "G70";
	\ar@{-} "B70";"G70";
	\ar@{} "B70";"D70";
	\ar@{} "B70";"E70";
	\ar@{} "B70";"F70";
	\ar@{-} "C70";"E70";
	\ar@{} "C70";"F70";
	\ar@{-} "C70";"G70";
	\ar@{} "D70";"F70";
	\ar@{} "D70";"G70";
	\ar@{-} "E70";"G70";

	\ar@{} (0,  0);(22,  -206) *\dir<2pt>{*} = "B71";
	\ar@{-} "B71";(20,  -210) *\dir<2pt>{*} = "C71";
	\ar@{-} "C71";(22,  -214) *\dir<2pt>{*} = "D71";
	\ar@{-} "D71";(26,  -214) *\dir<2pt>{*} = "E71";
	\ar@{-} "E71";(28,  -210) *\dir<2pt>{*} = "F71";
	\ar@{-} "F71";(26,  -206) *\dir<2pt>{*} = "G71";
	\ar@{-} "B71";"G71";
	\ar@{-} "B71";"D71";
	\ar@{} "B71";"E71";
	\ar@{} "B71";"F71";
	\ar@{} "C71";"E71";
	\ar@{-} "C71";"F71";
	\ar@{} "C71";"G71";
	\ar@{} "D71";"F71";
	\ar@{} "D71";"G71";
	\ar@{-} "E71";"G71";

	\ar@{} (0,  0);(34,  -206) *\dir<2pt>{*} = "B80";
	\ar@{-} "B80";(32,  -210) *\dir<2pt>{*} = "C80";
	\ar@{-} "C80";(34,  -214) *\dir<2pt>{*} = "D80";
	\ar@{-} "D80";(38,  -214) *\dir<2pt>{*} = "E80";
	\ar@{-} "E80";(40,  -210) *\dir<2pt>{*} = "F80";
	\ar@{-} "F80";(38,  -206) *\dir<2pt>{*} = "G80";
	\ar@{-} "B80";"G80";
	\ar@{-} "B80";"D80";
	\ar@{-} "B80";"E80";
	\ar@{} "B80";"F80";
	\ar@{} "C80";"E80";
	\ar@{} "C80";"F80";
	\ar@{-} "C80";"G80";
	\ar@{} "D80";"F80";
	\ar@{} "D80";"G80";
	\ar@{} "E80";"G80";

	\ar@{} (0,  0);(46,  -206) *\dir<2pt>{*} = "B94";
	\ar@{-} "B94";(44,  -210) *\dir<2pt>{*} = "C94";
	\ar@{-} "C94";(46,  -214) *\dir<2pt>{*} = "D94";
	\ar@{-} "D94";(50,  -214) *\dir<2pt>{*} = "E94";
	\ar@{-} "E94";(52,  -210) *\dir<2pt>{*} = "F94";
	\ar@{-} "F94";(50,  -206) *\dir<2pt>{*} = "G94";
	\ar@{-} "B94";"G94";
	\ar@{} "B94";"D94";
	\ar@{-} "B94";"E94";
	\ar@{} "B94";"F94";
	\ar@{-} "C94";"E94";
	\ar@{} "C94";"F94";
	\ar@{} "C94";"G94";
	\ar@{-} "D94";"F94";
	\ar@{} "D94";"G94";
	\ar@{-} "E94";"G94";

	\ar@{} (0,  0);(104,  -36) *\dir<2pt>{*} = "B20";
	\ar@{-} "B20";(102,  -40) *\dir<2pt>{*} = "C20";
	\ar@{-} "C20";(104,  -44) *\dir<2pt>{*} = "D20";
	\ar@{-} "D20";(108,  -44) *\dir<2pt>{*} = "E20";
	\ar@{-} "E20";(110,  -40) *\dir<2pt>{*} = "F20";
	\ar@{-} "F20";(108,  -36) *\dir<2pt>{*} = "G20";
	\ar@{} "B20";"G20";
	\ar@{-} "B20";"D20";
	\ar@{-} "B20";"E20";
	\ar@{} "B20";"F20";
	\ar@{} "C20";"E20";
	\ar@{} "C20";"F20";
	\ar@{} "C20";"G20";
	\ar@{} "D20";"F20";
	\ar@{} "D20";"G20";
	\ar@{-} "E20";"G20";

	\ar@{} (0,  0);(104,  -48) *\dir<2pt>{*} = "B24";
	\ar@{-} "B24";(102,  -52) *\dir<2pt>{*} = "C24";
	\ar@{-} "C24";(104,  -56) *\dir<2pt>{*} = "D24";
	\ar@{-} "D24";(108,  -56) *\dir<2pt>{*} = "E24";
	\ar@{-} "E24";(110,  -52) *\dir<2pt>{*} = "F24";
	\ar@{-} "F24";(108,  -48) *\dir<2pt>{*} = "G24";
	\ar@{} "B24";"G24";
	\ar@{-} "B24";"D24";
	\ar@{} "B24";"E24";
	\ar@{} "B24";"F24";
	\ar@{} "C24";"E24";
	\ar@{} "C24";"F24";
	\ar@{} "C24";"G24";
	\ar@{} "D24";"F24";
	\ar@{} "D24";"G24";
	\ar@{-} "E24";"G24";

	\ar@{} (0,  0);(104,  -60) *\dir<2pt>{*} = "B40";
	\ar@{} "B40";(102,  -64) *\dir<2pt>{*} = "C40";
	\ar@{-} "C40";(104,  -68) *\dir<2pt>{*} = "D40";
	\ar@{-} "D40";(108,  -68) *\dir<2pt>{*} = "E40";
	\ar@{-} "E40";(110,  -64) *\dir<2pt>{*} = "F40";
	\ar@{} "F40";(108,  -60) *\dir<2pt>{*} = "G40";
	\ar@{-} "B40";"G40";
	\ar@{-} "B40";"D40";
	\ar@{-} "B40";"E40";
	\ar@{} "B40";"F40";
	\ar@{} "C40";"E40";
	\ar@{} "C40";"F40";
	\ar@{} "C40";"G40";
	\ar@{} "D40";"F40";
	\ar@{-} "D40";"G40";
	\ar@{-} "E40";"G40";

	\ar@{} (0,  0);(104,  -72) *\dir<2pt>{*} = "B41";
	\ar@{-} "B41";(102,  -76) *\dir<2pt>{*} = "C41";
	\ar@{-} "C41";(104,  -80) *\dir<2pt>{*} = "D41";
	\ar@{-} "D41";(108,  -80) *\dir<2pt>{*} = "E41";
	\ar@{} "E41";(110,  -76) *\dir<2pt>{*} = "F41";
	\ar@{-} "F41";(108,  -72) *\dir<2pt>{*} = "G41";
	\ar@{-} "B41";"G41";
	\ar@{-} "B41";"D41";
	\ar@{-} "B41";"E41";
	\ar@{} "B41";"F41";
	\ar@{-} "C41";"E41";
	\ar@{} "C41";"F41";
	\ar@{} "C41";"G41";
	\ar@{} "D41";"F41";
	\ar@{} "D41";"G41";
	\ar@{} "E41";"G41";

	\ar@{} (0,  0);(104,  -84) *\dir<2pt>{*} = "B54";
	\ar@{-} "B54";(102,  -88) *\dir<2pt>{*} = "C54";
	\ar@{-} "C54";(104,  -92) *\dir<2pt>{*} = "D54";
	\ar@{-} "D54";(108,  -92) *\dir<2pt>{*} = "E54";
	\ar@{-} "E54";(110,  -88) *\dir<2pt>{*} = "F54";
	\ar@{-} "F54";(108,  -84) *\dir<2pt>{*} = "G54";
	\ar@{-} "B54";"G54";
	\ar@{-} "B54";"D54";
	\ar@{} "B54";"E54";
	\ar@{} "B54";"F54";
	\ar@{} "C54";"E54";
	\ar@{} "C54";"F54";
	\ar@{} "C54";"G54";
	\ar@{} "D54";"F54";
	\ar@{} "D54";"G54";
	\ar@{-} "E54";"G54";

	\ar@{} (0,  0);(104,  -96) *\dir<2pt>{*} = "B61";
	\ar@{-} "B61";(102,  -100) *\dir<2pt>{*} = "C61";
	\ar@{-} "C61";(104,  -104) *\dir<2pt>{*} = "D61";
	\ar@{-} "D61";(108,  -104) *\dir<2pt>{*} = "E61";
	\ar@{-} "E61";(110,  -100) *\dir<2pt>{*} = "F61";
	\ar@{} "F61";(108,  -96) *\dir<2pt>{*} = "G61";
	\ar@{-} "B61";"G61";
	\ar@{-} "B61";"D61";
	\ar@{-} "B61";"E61";
	\ar@{} "B61";"F61";
	\ar@{} "C61";"E61";
	\ar@{} "C61";"F61";
	\ar@{} "C61";"G61";
	\ar@{} "D61";"F61";
	\ar@{-} "D61";"G61";
	\ar@{-} "E61";"G61";

	\ar@{} (0,  0);(104,  -108) *\dir<2pt>{*} = "B62";
	\ar@{-} "B62";(102,  -112) *\dir<2pt>{*} = "C62";
	\ar@{-} "C62";(104,  -116) *\dir<2pt>{*} = "D62";
	\ar@{-} "D62";(108,  -116) *\dir<2pt>{*} = "E62";
	\ar@{-} "E62";(110,  -112) *\dir<2pt>{*} = "F62";
	\ar@{-} "F62";(108,  -108) *\dir<2pt>{*} = "G62";
	\ar@{-} "B62";"G62";
	\ar@{-} "B62";"D62";
	\ar@{} "B62";"E62";
	\ar@{} "B62";"F62";
	\ar@{} "C62";"E62";
	\ar@{} "C62";"F62";
	\ar@{} "C62";"G62";
	\ar@{} "D62";"F62";
	\ar@{-} "D62";"G62";
	\ar@{-} "E62";"G62";

	\ar@{} (0,  0);(104,  -120) *\dir<2pt>{*} = "B64";
	\ar@{-} "B64";(102,  -124) *\dir<2pt>{*} = "C64";
	\ar@{-} "C64";(104,  -128) *\dir<2pt>{*} = "D64";
	\ar@{-} "D64";(108,  -128) *\dir<2pt>{*} = "E64";
	\ar@{-} "E64";(110,  -124) *\dir<2pt>{*} = "F64";
	\ar@{} "F64";(108,  -120) *\dir<2pt>{*} = "G64";
	\ar@{-} "B64";"G64";
	\ar@{-} "B64";"D64";
	\ar@{} "B64";"E64";
	\ar@{} "B64";"F64";
	\ar@{} "C64";"E64";
	\ar@{} "C64";"F64";
	\ar@{-} "C64";"G64";
	\ar@{} "D64";"F64";
	\ar@{-} "D64";"G64";
	\ar@{-} "E64";"G64";

	\ar@{} (0,  0);(104,  -132) *\dir<2pt>{*} = "B77";
	\ar@{-} "B77";(102,  -136) *\dir<2pt>{*} = "C77";
	\ar@{-} "C77";(104,  -140) *\dir<2pt>{*} = "D77";
	\ar@{-} "D77";(108,  -140) *\dir<2pt>{*} = "E77";
	\ar@{-} "E77";(110,  -136) *\dir<2pt>{*} = "F77";
	\ar@{-} "F77";(108,  -132) *\dir<2pt>{*} = "G77";
	\ar@{-} "B77";"G77";
	\ar@{-} "B77";"D77";
	\ar@{} "B77";"E77";
	\ar@{} "B77";"F77";
	\ar@{} "C77";"E77";
	\ar@{} "C77";"F77";
	\ar@{-} "C77";"G77";
	\ar@{} "D77";"F77";
	\ar@{} "D77";"G77";
	\ar@{-} "E77";"G77";

	\ar@{} (0,  0);(104,  -144) *\dir<2pt>{*} = "B79";
	\ar@{-} "B79";(102,  -148) *\dir<2pt>{*} = "C79";
	\ar@{-} "C79";(104,  -152) *\dir<2pt>{*} = "D79";
	\ar@{-} "D79";(108,  -152) *\dir<2pt>{*} = "E79";
	\ar@{-} "E79";(110,  -148) *\dir<2pt>{*} = "F79";
	\ar@{-} "F79";(108,  -144) *\dir<2pt>{*} = "G79";
	\ar@{-} "B79";"G79";
	\ar@{-} "B79";"D79";
	\ar@{} "B79";"E79";
	\ar@{} "B79";"F79";
	\ar@{} "C79";"E79";
	\ar@{} "C79";"F79";
	\ar@{-} "C79";"G79";
	\ar@{} "D79";"F79";
	\ar@{-} "D79";"G79";
	\ar@{} "E79";"G79";

	\ar@{} (0,  0);(104,  -156) *\dir<2pt>{*} = "B81";
	\ar@{-} "B81";(102,  -160) *\dir<2pt>{*} = "C81";
	\ar@{-} "C81";(104,  -164) *\dir<2pt>{*} = "D81";
	\ar@{-} "D81";(108,  -164) *\dir<2pt>{*} = "E81";
	\ar@{-} "E81";(110,  -160) *\dir<2pt>{*} = "F81";
	\ar@{-} "F81";(108,  -156) *\dir<2pt>{*} = "G81";
	\ar@{-} "B81";"G81";
	\ar@{-} "B81";"D81";
	\ar@{-} "B81";"E81";
	\ar@{} "B81";"F81";
	\ar@{} "C81";"E81";
	\ar@{} "C81";"F81";
	\ar@{} "C81";"G81";
	\ar@{} "D81";"F81";
	\ar@{-} "D81";"G81";
	\ar@{-} "E81";"G81";

	\ar@{} (0,  0);(104,  -168) *\dir<2pt>{*} = "B82";
	\ar@{-} "B82";(102,  -172) *\dir<2pt>{*} = "C82";
	\ar@{-} "C82";(104,  -176) *\dir<2pt>{*} = "D82";
	\ar@{-} "D82";(108,  -176) *\dir<2pt>{*} = "E82";
	\ar@{-} "E82";(110,  -172) *\dir<2pt>{*} = "F82";
	\ar@{-} "F82";(108,  -168) *\dir<2pt>{*} = "G82";
	\ar@{-} "B82";"G82";
	\ar@{-} "B82";"D82";
	\ar@{-} "B82";"E82";
	\ar@{} "B82";"F82";
	\ar@{-} "C82";"E82";
	\ar@{} "C82";"F82";
	\ar@{} "C82";"G82";
	\ar@{} "D82";"F82";
	\ar@{} "D82";"G82";
	\ar@{-} "E82";"G82";

	\ar@{} (0,  0);(116,  -36) *\dir<2pt>{*} = "B83";
	\ar@{-} "B83";(114,  -40) *\dir<2pt>{*} = "C83";
	\ar@{-} "C83";(116,  -44) *\dir<2pt>{*} = "D83";
	\ar@{-} "D83";(120,  -44) *\dir<2pt>{*} = "E83";
	\ar@{-} "E83";(122,  -40) *\dir<2pt>{*} = "F83";
	\ar@{} "F83";(120,  -36) *\dir<2pt>{*} = "G83";
	\ar@{-} "B83";"G83";
	\ar@{-} "B83";"D83";
	\ar@{} "B83";"E83";
	\ar@{} "B83";"F83";
	\ar@{-} "C83";"E83";
	\ar@{} "C83";"F83";
	\ar@{-} "C83";"G83";
	\ar@{} "D83";"F83";
	\ar@{-} "D83";"G83";
	\ar@{-} "E83";"G83";

	\ar@{} (0,  0);(116,  -48) *\dir<2pt>{*} = "B84";
	\ar@{-} "B84";(114,  -52) *\dir<2pt>{*} = "C84";
	\ar@{-} "C84";(116,  -56) *\dir<2pt>{*} = "D84";
	\ar@{-} "D84";(120,  -56) *\dir<2pt>{*} = "E84";
	\ar@{-} "E84";(122,  -52) *\dir<2pt>{*} = "F84";
	\ar@{-} "F84";(120,  -48) *\dir<2pt>{*} = "G84";
	\ar@{-} "B84";"G84";
	\ar@{} "B84";"D84";
	\ar@{-} "B84";"E84";
	\ar@{} "B84";"F84";
	\ar@{-} "C84";"E84";
	\ar@{} "C84";"F84";
	\ar@{-} "C84";"G84";
	\ar@{} "D84";"F84";
	\ar@{} "D84";"G84";
	\ar@{-} "E84";"G84";

	\ar@{} (0,  0);(116,  -60) *\dir<2pt>{*} = "B88";
	\ar@{-} "B88";(114,  -64) *\dir<2pt>{*} = "C88";
	\ar@{-} "C88";(116,  -68) *\dir<2pt>{*} = "D88";
	\ar@{-} "D88";(120,  -68) *\dir<2pt>{*} = "E88";
	\ar@{-} "E88";(122,  -64) *\dir<2pt>{*} = "F88";
	\ar@{-} "F88";(120,  -60) *\dir<2pt>{*} = "G88";
	\ar@{-} "B88";"G88";
	\ar@{} "B88";"D88";
	\ar@{} "B88";"E88";
	\ar@{-} "B88";"F88";
	\ar@{-} "C88";"E88";
	\ar@{} "C88";"F88";
	\ar@{-} "C88";"G88";
	\ar@{-} "D88";"F88";
	\ar@{} "D88";"G88";
	\ar@{} "E88";"G88";

	\ar@{} (0,  0);(116,  -72) *\dir<2pt>{*} = "B90";
	\ar@{-} "B90";(114,  -76) *\dir<2pt>{*} = "C90";
	\ar@{-} "C90";(116,  -80) *\dir<2pt>{*} = "D90";
	\ar@{-} "D90";(120,  -80) *\dir<2pt>{*} = "E90";
	\ar@{-} "E90";(122,  -76) *\dir<2pt>{*} = "F90";
	\ar@{-} "F90";(120,  -72) *\dir<2pt>{*} = "G90";
	\ar@{-} "B90";"G90";
	\ar@{} "B90";"D90";
	\ar@{} "B90";"E90";
	\ar@{-} "B90";"F90";
	\ar@{-} "C90";"E90";
	\ar@{} "C90";"F90";
	\ar@{} "C90";"G90";
	\ar@{-} "D90";"F90";
	\ar@{} "D90";"G90";
	\ar@{-} "E90";"G90";

	\ar@{} (0,  0);(116,  -84) *\dir<2pt>{*} = "B92";
	\ar@{-} "B92";(114,  -88) *\dir<2pt>{*} = "C92";
	\ar@{-} "C92";(116,  -92) *\dir<2pt>{*} = "D92";
	\ar@{-} "D92";(120,  -92) *\dir<2pt>{*} = "E92";
	\ar@{-} "E92";(122,  -88) *\dir<2pt>{*} = "F92";
	\ar@{-} "F92";(120,  -84) *\dir<2pt>{*} = "G92";
	\ar@{-} "B92";"G92";
	\ar@{} "B92";"D92";
	\ar@{-} "B92";"E92";
	\ar@{} "B92";"F92";
	\ar@{-} "C92";"E92";
	\ar@{-} "C92";"F92";
	\ar@{} "C92";"G92";
	\ar@{-} "D92";"F92";
	\ar@{} "D92";"G92";
	\ar@{} "E92";"G92";

	\ar@{} (0,  0);(116,  -96) *\dir<2pt>{*} = "B97";
	\ar@{-} "B97";(114,  -100) *\dir<2pt>{*} = "C97";
	\ar@{-} "C97";(116,  -104) *\dir<2pt>{*} = "D97";
	\ar@{-} "D97";(120,  -104) *\dir<2pt>{*} = "E97";
	\ar@{-} "E97";(122,  -100) *\dir<2pt>{*} = "F97";
	\ar@{-} "F97";(120,  -96) *\dir<2pt>{*} = "G97";
	\ar@{} "B97";"G97";
	\ar@{-} "B97";"D97";
	\ar@{-} "B97";"E97";
	\ar@{-} "B97";"F97";
	\ar@{} "C97";"E97";
	\ar@{-} "C97";"F97";
	\ar@{-} "C97";"G97";
	\ar@{-} "D97";"F97";
	\ar@{} "D97";"G97";
	\ar@{} "E97";"G97";

	\ar@{} (0,  0);(116,  -108) *\dir<2pt>{*} = "B101";
	\ar@{-} "B101";(114,  -112) *\dir<2pt>{*} = "C101";
	\ar@{-} "C101";(116,  -116) *\dir<2pt>{*} = "D101";
	\ar@{-} "D101";(120,  -116) *\dir<2pt>{*} = "E101";
	\ar@{-} "E101";(122,  -112) *\dir<2pt>{*} = "F101";
	\ar@{-} "F101";(120,  -108) *\dir<2pt>{*} = "G101";
	\ar@{-} "B101";"G101";
	\ar@{-} "B101";"D101";
	\ar@{-} "B101";"E101";
	\ar@{} "B101";"F101";
	\ar@{} "C101";"E101";
	\ar@{-} "C101";"F101";
	\ar@{} "C101";"G101";
	\ar@{} "D101";"F101";
	\ar@{-} "D101";"G101";
	\ar@{-} "E101";"G101";

	\ar@{} (0,  0);(116,  -120) *\dir<2pt>{*} = "B102";
	\ar@{-} "B102";(114,  -124) *\dir<2pt>{*} = "C102";
	\ar@{-} "C102";(116,  -128) *\dir<2pt>{*} = "D102";
	\ar@{-} "D102";(120,  -128) *\dir<2pt>{*} = "E102";
	\ar@{-} "E102";(122,  -124) *\dir<2pt>{*} = "F102";
	\ar@{-} "F102";(120,  -120) *\dir<2pt>{*} = "G102";
	\ar@{} "B102";"G102";
	\ar@{-} "B102";"D102";
	\ar@{-} "B102";"E102";
	\ar@{-} "B102";"F102";
	\ar@{-} "C102";"E102";
	\ar@{} "C102";"F102";
	\ar@{-} "C102";"G102";
	\ar@{-} "D102";"F102";
	\ar@{} "D102";"G102";
	\ar@{} "E102";"G102";

	\ar@{} (0,  0);(116,  -132) *\dir<2pt>{*} = "B103";
	\ar@{-} "B103";(114,  -136) *\dir<2pt>{*} = "C103";
	\ar@{-} "C103";(116,  -140) *\dir<2pt>{*} = "D103";
	\ar@{-} "D103";(120,  -140) *\dir<2pt>{*} = "E103";
	\ar@{-} "E103";(122,  -136) *\dir<2pt>{*} = "F103";
	\ar@{-} "F103";(120,  -132) *\dir<2pt>{*} = "G103";
	\ar@{-} "B103";"G103";
	\ar@{} "B103";"D103";
	\ar@{} "B103";"E103";
	\ar@{-} "B103";"F103";
	\ar@{-} "C103";"E103";
	\ar@{-} "C103";"F103";
	\ar@{} "C103";"G103";
	\ar@{-} "D103";"F103";
	\ar@{-} "D103";"G103";
	\ar@{} "E103";"G103";

	\ar@{} (0,  0);(116,  -144) *\dir<2pt>{*} = "B104";
	\ar@{-} "B104";(114,  -148) *\dir<2pt>{*} = "C104";
	\ar@{-} "C104";(116,  -152) *\dir<2pt>{*} = "D104";
	\ar@{-} "D104";(120,  -152) *\dir<2pt>{*} = "E104";
	\ar@{-} "E104";(122,  -148) *\dir<2pt>{*} = "F104";
	\ar@{-} "F104";(120,  -144) *\dir<2pt>{*} = "G104";
	\ar@{-} "B104";"G104";
	\ar@{} "B104";"D104";
	\ar@{-} "B104";"E104";
	\ar@{-} "B104";"F104";
	\ar@{-} "C104";"E104";
	\ar@{-} "C104";"F104";
	\ar@{-} "C104";"G104";
	\ar@{-} "D104";"F104";
	\ar@{} "D104";"G104";
	\ar@{} "E104";"G104";

	\ar@{} (0,  0);(116,  -156) *\dir<2pt>{*} = "B107";
	\ar@{-} "B107";(114,  -160) *\dir<2pt>{*} = "C107";
	\ar@{-} "C107";(116,  -164) *\dir<2pt>{*} = "D107";
	\ar@{-} "D107";(120,  -164) *\dir<2pt>{*} = "E107";
	\ar@{-} "E107";(122,  -160) *\dir<2pt>{*} = "F107";
	\ar@{-} "F107";(120,  -156) *\dir<2pt>{*} = "G107";
	\ar@{-} "B107";"G107";
	\ar@{} "B107";"D107";
	\ar@{-} "B107";"E107";
	\ar@{} "B107";"F107";
	\ar@{-} "C107";"E107";
	\ar@{-} "C107";"F107";
	\ar@{} "C107";"G107";
	\ar@{-} "D107";"F107";
	\ar@{-} "D107";"G107";
	\ar@{-} "E107";"G107";

	\ar@{} (0,  0);(116,  -168) *\dir<2pt>{*} = "B110";
	\ar@{-} "B110";(114,  -172) *\dir<2pt>{*} = "C110";
	\ar@{-} "C110";(116,  -176) *\dir<2pt>{*} = "D110";
	\ar@{-} "D110";(120,  -176) *\dir<2pt>{*} = "E110";
	\ar@{-} "E110";(122,  -172) *\dir<2pt>{*} = "F110";
	\ar@{-} "F110";(120,  -168) *\dir<2pt>{*} = "G110";
	\ar@{-} "B110";"G110";
	\ar@{} "B110";"D110";
	\ar@{-} "B110";"E110";
	\ar@{-} "B110";"F110";
	\ar@{-} "C110";"E110";
	\ar@{-} "C110";"F110";
	\ar@{-} "C110";"G110";
	\ar@{-} "D110";"F110";
	\ar@{-} "D110";"G110";
	\ar@{} "E110";"G110";

	\textcolor{red}{\ar@{} (0,  0);(-13,  -32);}
	\textcolor{red}{\ar@{-} (-13,  -32);(2,  -32);}
	\textcolor{red}{\ar@{} (0,  0);(20,  -32) *\txt{HL-Comparability};}
	\textcolor{red}{\ar@{-} (37,  -32);(124,  -32);}
	\textcolor{red}{\ar@{-} (-17,  -32);(-17,  -186);}
	\textcolor{red}{\ar@{-} (121,  -32);(121,  -186);}
	\textcolor{red}{\ar@{-} (-20,  -186);(120,  -186);}

	\textcolor{black}{\ar@{} (0,  0);(-17,  -40);}
	\textcolor{black}{\ar@{-} (-17,  -40);(12,  -40);}
	\textcolor{black}{\ar@{} (0,  0);(20,  -40) *\txt{Bipartite};}
	\textcolor{black}{\ar@{-} (29,  -40);(55,  -40);}
	\textcolor{black}{\ar@{-} (-21,  -40);(-21,  -82);}
	\textcolor{black}{\ar@{-} (52,  -40);(52,  -82);}
	\textcolor{black}{\ar@{-} (-24,  -82);(51,  -82);}

	\textcolor{cyan}{\ar@{} (0,  0);(35,  -68);}
	\textcolor{cyan}{\ar@{-} (35,  -68);(78,  -68);}
	\textcolor{cyan}{\ar@{} (0,  0);(62,  -60) *\txt{Trivially};}
	\textcolor{cyan}{\ar@{} (0,  0);(60,  -64) *\txt{Perfect};}
	\textcolor{cyan}{\ar@{-} (31,  -68);(31,  -158);}
	\textcolor{cyan}{\ar@{-} (72.5,  -68);(72.5,  -158);}
	\textcolor{cyan}{\ar@{-} (28,  -158);(71,  -158);}

\end{xy}


\par \vspace{2mm}
${\mb{Acknowledgement.}}$ I wish to thank Professor Hidefumi Ohsugi
for many valuable comments. 
This research was supported by the JST (Japan Science and Technology Agency)
CREST (Core Research for Evolutional Science and Technology) research project
Harmony of Gr\"{o}bner Bases and the Modern Industrial Society in the framework of the
JST Mathematics Program ``Alliance for Breakthrough between Mathematics and
Sciences."


\end{document}